\documentclass[onefignum,onetabnum]{siamart220329}
\usepackage{mathtools}

\usepackage{footnotebackref}
\usepackage{xcolor}
\usepackage{amsmath,amssymb,amsfonts}

\usepackage{rotating}
\usepackage{multicol,array}
\usepackage{multirow}
\usepackage{longtable}

\usepackage{subcaption}

\usepackage{url}
\usepackage{algorithm,algorithmic} 
\usepackage{mathtools}
\usepackage{tcolorbox}
\usepackage{hyperref}
\usepackage{bbm}
\usepackage{todonotes}
 \usepackage{caption}
 \usepackage{subcaption}
\usepackage{graphicx}
\usepackage{rotating}

\newcommand{\real}{\mathbb{R}}

\newcommand{\RR}{\mathbb{R}}
\newcommand{\R}{\mathbb{R}}

\usepackage{accents}

\newcommand{\cN}{{\cal N}}
\newcommand{\Dx}{\Delta x}
\newcommand{\Dv}{\Delta v}
\newcommand{\Dlambda}{\Delta \lambda}
\newcommand{\Ds}{\Delta s}

\newcommand{\epspf}{\epsilon_{\rm pf}}
\newcommand{\epsfoc}{\epsilon_{\rm foc}}
\newcommand{\epssoc}{\epsilon_{\rm soc}}
\newcommand{\epscs}{\epsilon_{\rm cs}}
\newcommand{\epspd}{\epsilon_{\rm pd}}
\newcommand{\alphac}{\zeta}

\mathchardef\mhyphen="2D 

\DeclarePairedDelimiter{\norm}{\lVert}{\rVert}

\DeclareMathOperator{\nullspace}{null}
\DeclareMathOperator{\diag}{diag}

\theoremstyle{remark}
\newtheorem{remark}{Remark}[section]

\newcommand{\Amap}{\mathcal{A}}

\newcommand{\dm}{d}

\newcommand{\twonorm}[1]{\left\|#1\right\|_2}

\newcommand{\onenorm}[1]{\left\|#1\right\|_{1}}
\newcommand{\infnorm}[1]{\left\|#1\right\|_{\infty}}
\newcommand{\fronorm}[1]{\left\|#1\right\|_F}

\let\originalleft\left
\let\originalright\right
\renewcommand{\left}{\mathopen{}\mathclose\bgroup\originalleft}
\renewcommand{\right}{\aftergroup\egroup\originalright}

\usepackage{multirow}

\newcommand{\cA}{{\cal A}}
\newcommand{\cI}{{\cal I}}
\newcommand{\cJ}{{\overline{\cal I}}}
\newcommand{\cE}{{\cal E}}
\newcommand{\cD}{{\cal D}}
\newcommand{\cL}{{\cal L}}
\newcommand{\JSSV}{J_{\text{SSV}}}

\newcommand{\editContent}[1]{\color{black} #1 \color{black}}
\newcommand{\rev}[1]{{\color{black}#1} \color{black}}
\newcommand{\sw}[1]{{\color{purple}{[\textbf{Steve:} #1}]}}
\newcommand{\ld}[1]{{\color{red}{[\textbf{Lijun:} #1}]}}

\headers{On Squared-Variable Formulations}{L. Ding and S. J. Wright}

\title{On Squared-Variable Formulations}

\author{Lijun Ding\thanks{Department of Mathematics, University of California San Diego, La Jolla, CA 92093, USA (\email{l2ding@ucsd.edu}).}
\and Stephen J. Wright\thanks{Department of Computer Sciences,
University of Wisconsin, Madison, WI 53706, USA (\email{swright@cs.wisc.edu}).
}
}

\begin{document}

\maketitle

\begin{abstract}
We revisit a formulation technique for inequality constrained optimization problems that has been known for decades: the substitution of squared variables for nonnegative variables.
Using this technique, inequality constraints are converted to equality constraints via the introduction of a squared-slack variable.
Such formulations have the superficial advantage that inequality constraints can be dispensed with altogether.
But there are clear disadvantages, not least being that first-order optimal points for the squared-variable reformulation may not correspond to first-order optimal points for the original problem, because the Lagrange multipliers may have the wrong sign. 
Extending previous results, this paper shows that  points satisfying approximate second-order  optimality conditions for the squared-variable reformulation also, under certain conditions, satisfy approximate second-order optimality conditions for the original formulation, and vice versa. 
Such results allow us to adapt complexity analysis of methods for equality constrained optimization to account for inequality constraints.
On the algorithmic side, we examine squared-variable formulations for several interesting problem classes, including bound-constrained quadratic programming, linear programming, and nonnegative matrix factorization. 
We show that algorithms built on these formulations are surprisingly competitive with standard methods.
For linear programming, we examine the relationship between the squared-variable approach and primal-dual interior-point methods.
\end{abstract}


\section{Introduction} \label{sec:intro}

Consider a smooth objective $f:\real^n \rightarrow \real$ and smooth vector function $c: \real^n \rightarrow \real^m$ that define an inequality constrained optimization problem
\begin{equation}
    \label{eq:f} \tag{NLP}
    \min_x \, f(x) \;\; \mbox{s.t.} \;\; c(x) \geq 0.
\end{equation}
A formulation of this problem using squared-slack variables (SSV) is
\begin{equation}
    \label{eq:ssv} \tag{SSV}
    \min_{x,v} \, f(x) \;\; \mbox{s.t.} \;\; c(x) - v\odot v=0,
\end{equation}
where $v \in \R^m$. 
Relatedly, we can use squared variables in place of nonnegative variables. 
Consider the problem
\begin{equation}
    \label{eq:bc} \tag{BC}
    \min_x \, f(x) \;\; \mbox{s.t.} \;\; x \ge 0.
\end{equation}
By directly substituting $x = v \odot v$, we obtain the unconstrained problem
\begin{equation} \label{eq:bc.ssv} \tag{DSS}
\min_v \, F(v) := f(v \odot v).
\end{equation}

Squared-slack variables have a long history in optimization as a means of converting a problem with inequality constraints into one with equality constraints, thereby avoiding the issues involved in determining which inequality constraints are active at the solution.
Despite this superficial appeal, there are potentially serious objections to the approach, including the following.
\begin{enumerate}
\item A first-order optimal point for \eqref{eq:ssv} may not yield a first-order point for \eqref{eq:f}, because Lagrange multipliers for the constraints in \eqref{eq:ssv} may have either sign, whereas Lagrange multipliers in \eqref{eq:f} must be nonnegative.
\item Any solution $x^*$ of \eqref{eq:f} corresponds to a multiplicity of solutions of \eqref{eq:ssv}, obtained by setting $v_i^* = \pm \sqrt{c_i(x^*)}$. That is, the sign of $v_i$ is not well defined.
\item Convexity of the original formulation is generally lost in the squared-variable formulations. 
\item The reformulated problem is ``more nonlinear." Linear inequality constraints become quadratic equality constraints, quadratic objectives may become quartic, and so on.
\item Certain issues involving conditioning and numerical stability (that depend on the choice of algorithm) arise in \eqref{eq:ssv} that do not appear in \eqref{eq:f}.
\end{enumerate}

Early papers of Tapia \cite{tapia74a,tapia80a} described Newton and quasi-Newton methods applied to the first-order optimality conditions for \eqref{eq:ssv}, but while mentioning the issue concerning the sign of the Lagrange multipliers, did not propose a strategy to address it.
Bertsekas~\cite[Section~3.3.2]{Ber99} uses the squared-slack formulation for theoretical purposes, in deriving optimality conditions for inequality-constrained optimization.  
He shows that when the  primal-dual solution $(x^*,v^*,s^*)$ of \eqref{eq:ssv} ($s^*$ being the Lagrange multiplier for the equality constraint) satisfies second-order {\em necessary} conditions with $v^*\geq 0$, then $(x^*,s^*)$ satisfies the first- and a weak form of second-order necessary conditions for \eqref{eq:f}.
(In fact, it is trivial to drop the restriction $v^*\geq 0$.)
Fukuda and Fukushima~\cite{fukuda2017note} show that when the primal-dual solution $(x^*,v^*,s^*)$ of \eqref{eq:ssv} satisfies second-order {\em sufficient} conditions, then $(x^*,s^*)$ satisfies second-order sufficient conditions for \eqref{eq:f}. 
They also discuss the relationship between LICQ in the two formulations.

The direct-square-substitution approach \eqref{eq:bc.ssv} for the bound-constrained problem \eqref{eq:bc} is discussed in Gill, Murray, and Wright's text \cite[pp.~268-269]{GilMW81}. 
They point out (as we discuss in \Cref{sec:bound}) that spurious saddle points for the reformulated problem arise when some components of $v$ are zero. 

The paper \cite{li2021simplex} considers optimization over the simplex, in which the constraint $\mathbf{1}_n^Tx = 1$ is added to the nonnegativity constraints in \eqref{eq:bc}.
The paper explores the reformulation \eqref{eq:bc.ssv}, with the constraint $\sum_{i=1}^n v_i^2=1$ included, and treats the resulting problem as one of minimization over the Riemannian manifold $\mathcal{S}_{n-1} := \{ v \, | \, \sum_{i=1}^n v_i^2 =1 \}$, referring to this reformulation technique as the ``Hadamard parametrization."
They apply a perturbed form of Riemannian gradient descent to this formulation, with various steplength selection strategies, proving convergence to a second-order optimal point of the original problem, with a certain complexity guarantee $\mathcal{O}(\log^4 n/\epsilon)$ (under assumptions of a strict saddle property, nondegeneracy, and no spurious local minima.) 
In  \cite{levin2022effect}, settings beyond the simplex constraint are considered, such as the set of low-rank matrices, and show landscape results concerning 
exact first and second-order necessary points.

\subsection*{Motivation, Outline, Contributions}
In the paper we examine several aspects of squared-variable formulations. 
First, we examine points that satisfy second-order optimality conditions {\em approximately}  for the direct-square-substitution and squared-slack formulations, and discuss whether such points satisfy approximate second-order optimality conditions for the original formulations. 
This connection allows us to perform a complexity analysis for the inequality-constrained problem \eqref{eq:f} by applying recent results on complexity of equality-constrained optimization to the formulation \eqref{eq:ssv}.
Second, 
we focus on a squared-variable formulation of linear programming, and draw connections between sequential quadratic programming applied to this formulation and the classical primal-dual interior-point method.
In a final section, we report preliminary numerical results for several formulations of interest, including bound-constrained convex quadratic programming, linear programming, \rev{nonnegative matrix factorization,} and least-squares problem with ``pseudo-norm" constraints (a problem class that arises in variable selection in linear statistical models and compressed sensing).

\subsection*{Notation}
The Hadamard product $\odot$ of two vectors $v$ and $w$ of identical length is defined as $v \odot w = [ v_i w_i]_{i=1}^n$. 
In particular, $v \odot v = [v_i^2]_{i=1}^n$. 
We use $\mathbf{1}_n$ to be the vector in $\R^n$ whose components are all $1$ \rev{ and $e_i$ to be the vector in $\R^n$  whose $i$th component is $1$ with all other components $0$.}
For a vector $v \in \R^n$, $\diag(v)$ is the $n \times n$ diagonal matrix whose $i$th diagonal element is $v_i$, $i=1,2,\dotsc,n$. 
We often abbreviate the term ``second-order necessary" by ``2N". A ``2N point" or ``2NP" is a point that satisfies 2N conditions.
\rev{A ``weak 2N point" or ``weak 2NP" is a point that satisfies weak 2N conditions.
A ``2S point" is a point that satisfies second-order sufficient conditions.}
We occasionally use ``1N" for first-order necessary conditions and ``1P" for  a point that satisfies first-order necessary conditions.

\section{Nonnegativity constraints} \label{sec:bound}

\subsection{Equivalence of 2N points for \eqref{eq:bc} and \eqref{eq:bc.ssv}} 
Here we discuss the relationship between optimality conditions for the original formulation and squared-variable reformulation of problems in which the variables are constrained to be nonnegative.

Consider the bound-constrained problem \eqref{eq:bc} and its reformulation \eqref{eq:bc.ssv}.
Using the notation $V = \diag(v)$, the gradients and Hessian of $F$ are, respectively,
\begin{subequations} \label{eq:Fopt}
\begin{align}
\label{eq:Fopt.1}
\nabla F(v) &= 2 V \nabla f(v \odot v) \rev{= 2 v \odot \nabla f(v \odot v),} \\
\label{eq:Fopt.2}
\nabla^2 F(v) &= 2 \diag(\nabla f(v \odot v)) + 4 V \nabla^2 f(v \odot v) V.
\end{align}
\end{subequations}

Optimality conditions for \eqref{eq:bc} are as follows.
\editContent{
\begin{definition} \label{def:bc}
We say that $x^*$ satisfies {\em first-order (1N) conditions} for \eqref{eq:bc} if it satisfies primal-dual feasibility $x^* \ge 0$, $\nabla f(x^*) \ge 0$ and complementarity \\ 
$x^*_i [\nabla f(x^*)]_i=0$, $i=1,2,\dotsc, n$.
Given $x^*$ satisfying such conditions, we define the following partition of $\{1,2,\dotsc,n\}$:
\begin{align*}
\mathcal{I} & = \{i\mid x_i^*>0\} & \; \mbox{(inactive indices)} \\
\mathcal{A} & = \{i \mid x^*_i=0, [\nabla f(x^*)]_i >0 \} & \; \mbox{(active indices)} \\
\mathcal{D} & = \{i \mid x^*_i=0, [\nabla f(x^*)]_i =0\}  & \; \mbox{(degenerate indices).} 
\end{align*}
The following terms are used in association with a point $x^*$ that satisfies first-order conditions and certain other conditions:
\begin{itemize}
\item {\em strict complementarity}: 1N conditions and $\mathcal{D}=\emptyset$.
\item {\rm degenerate first-order point}: 1N conditions and $\mathcal{D} \neq \emptyset$.
\item {\em weak second-order necessary (weak 2N) conditions}:  1N conditions and \\
$w^T \nabla^2 f(x^*) w \ge 0$ for all $w$ such that $w_i = 0$ for $i\in \mathcal{A}\cup \mathcal{D}$.
\item {\em second-order necessary (2N) conditions}: 1N conditions and 
$w^T \nabla^2 f(x^*) w \geq 0$ for all $w$ such that $w_i = 0$ for $i\in \mathcal{A}$ and $w_i\geq 0$ for $i\in \mathcal{D}$. 
\item {\em second-order sufficient (2S) conditions}: 1N conditions and 
$w^T \nabla^2 f(x^*) w > 0$ for all $w  \neq 0$ such that $w_i = 0$ for $i\in \mathcal{A}$, and $w_i\geq 0$ for $i\in \mathcal{D}$. 
\end{itemize}
\end{definition}
}


The second-order necessary conditions correspond to a copositivity condition in general \editContent{(for degenerate solutions, when $\mathcal{D}$ is nonempty, the inequality $w^T \nabla f(x^*)w \geq 0$ holds for all $w $ over a \emph{cone} rather than a \emph{subspace}),}  which is hard to \rev{verify.}
The {\em weak} second-order necessary conditions require positive semidefiniteness to be satisfied on a subspace of $\R^n$, which is easy to check. 

We turn next to optimality conditions for \eqref{eq:bc.ssv}.
\begin{definition} \label{def:bc.ssv}
We say that $v^*$ satisfies {\em first-order necessary (1N) } conditions to be a solution of  \eqref{eq:bc.ssv} when $\nabla F(v^*)=0$. 
We say that $v^*$ satisfies {\em second-order necessary (2N) } conditions to be a solution of  \eqref{eq:bc.ssv} if $\nabla F(v^*)=0$ and 
$\nabla^2 F(v^*) \succeq 0$.
We say that $v^*$ satisfies {\em second-order sufficient (2S) } conditions to be a solution of  \eqref{eq:bc.ssv} if $\nabla F(v^*)=0$ and  $\nabla^2 F(v^*) \succ 0$.
\end{definition}

\rev{For reference in the results below, by letting $x^* = v^* \odot v^*$,  we have from \eqref{eq:Fopt.2} and $x_i^*\not=0\iff v_i^*\not = 0$, for any $s \in \mathbb{R}^n$, that 
\begin{equation}\label{eq: q_form_F}
\begin{aligned}
s^T \nabla^2 F(v^*) s 
& = 
2 \sum_{i=1}^n s_i^2 [\nabla f(x^*)]_i + 4 \sum_{i,j \in \mathcal{I}} [\nabla^2 f(x^*)]_{ij}  (v^*_i s_i) (v^*_j s_j),
\end{aligned}
\end{equation}
where $\mathcal{I} = \{i\in [1,n]\mid x^*_i >0\}$ is the set of indices for inactive constraints.
}

The following result relates 2N  conditions for the formulations \eqref{eq:bc} and \eqref{eq:bc.ssv}.
\begin{theorem}
    \label{th:bc}
    $x^*$ satisfies weak second-order necessary (weak 2N) conditions for \eqref{eq:bc} if and only if $v^*$ defined by $v^*_i = \pm \sqrt{x^*_i}$ satisfies  second-order necessary (2N) conditions for \eqref{eq:bc.ssv}.
\end{theorem}
\rev{
\begin{proof}
Suppose that $x^*$ satisfies weak 2N conditions for \eqref{eq:bc}, and define $v^*$ as in the theorem. Since $x^*_i \neq 0 \iff v^*_i \neq 0$, the complementary conditions $x_i^* [\nabla f(x^*)]_i = 0$ for all $i=1,\dots,n$ imply that $\nabla F(v^*) = v^* \odot  \nabla f(v^* \odot v^*) = 0$. 
Given any $s \in \R^n$, we have from  \eqref{eq: q_form_F} that $s^T \nabla ^2 F(v^*)s\geq 0$,
because $\nabla f(x^*) \ge 0$ (so the first summation is nonnegative) and $w^\top \nabla^2 f(x^*) w = \sum_{i,j\in \mathcal{I}} [\nabla^2 f(x^*)]_{ij} w_iw_j \geq 0$ for all $w$ with $w_i=0$ for $i\in \mathcal{A}\cup \mathcal{D}$ (so the second summation is also nonnegative).
Thus 2N conditions hold for \eqref{eq:bc.ssv} hold at $v^*$.

Suppose now that 2N conditions hold for \eqref{eq:bc.ssv} hold at $v^*$, and set $x^* = v^* \odot v^*$. 
We have $x^* \ge 0$ and that the complementarity conditions 
$x^*_i [\nabla f(x^*)]_i =0$, $i=1,2,\dotsc,n$  hold because $0 = \nabla F(x^*) = v^*\odot  \nabla f(x^*) \implies  x^*\odot  \nabla f(x^*) =0$.
For any index $i$ such that $x^*_i>0$, we have $[\nabla f(x^*)]_i =0$ from complementarity. 
To complete the proof of dual feasibility, we need to show that $[\nabla f(x^*)]_i  \ge 0$ for all $i$ such that $x^*_i=0$.
By setting $s = e_i$ for any such $i$ in \eqref{eq: q_form_F}, and using $s_j=0$ for all $j \neq i$, we obtain $0 \le s^T \nabla^2 F(v^*) s = 2 [\nabla f(x^*)]_i$, as required. Thus $x^*$ is a 1N point for \eqref{eq:bc}, and we can define sets $\mathcal{A}$, $\mathcal{I}$, and $\mathcal{D}$ according to Definition~\ref{def:bc}.

To prove that $x^*$ is in fact a weak 2N point for \eqref{eq:bc}, we need to show that $w^T \nabla^2 f(x^*) w \ge 0$ for all $w$ such that $w_i=0$ for $i \in \mathcal{A} \cup \mathcal{D}$, that is, $i \notin  \mathcal{I}$.
We define $s \in \mathbb{R}^n$ by $s_i=w_i/v^*_i$ for $i \in \mathcal{I}$ and $s_i=0$ for $i \in \mathcal{A} \cup \mathcal{D}$, and note that $s_i [\nabla f(x^*)]_i=0$ for all $i = 1,2,\dotsc,n$ for this $s$. 
By substituting this vector $s$ in \eqref{eq: q_form_F}, we obtain 
$ 0\le s^T \nabla^2 F(v^*) s =  4 \sum_{i \in \mathcal{I}} \sum_{j \in \mathcal{I}}  [\nabla^2 f(x^*)]_{ij}  w_i w_j  = 4 w^T \nabla^2 f(x^*) w,$
so weak 2N conditions are satisfied for \eqref{eq:bc} at $x^*$.
\end{proof}
}

We show next a relationship between points satisfying second-order sufficient conditions for \eqref{eq:bc} and \eqref{eq:bc.ssv}.
\begin{theorem}
    \label{th:bc.suff}
    $x^*$ satisfies 2S conditions and strict complementarity for \eqref{eq:bc} if and only if $v^*$ defined by $v^*_i = \pm \sqrt{x^*_i}$ satisfies  2S conditions for \eqref{eq:bc.ssv}.
\end{theorem}
\rev{
\begin{proof}
Suppose that $x^*$ satisfies 2S conditions and strict complementarity ($\mathcal{D} = \emptyset$) for \eqref{eq:bc}, and define $v^*$ as in the theorem. 
From Theorem~\ref{th:bc}, we have that $v^*$ satisfies weak 2N conditions and thus 1N conditions. Given any $s \in \R^n$, from strict complementarity, we know 
$\sum_{i=1}^n s_i^2 [\nabla f(x^*)]_i = \sum_{i \in \mathcal{A}} s_i^2 [\nabla f(x^*)]_i$. The expression in \eqref{eq: q_form_F} then becomes 
\[
s^T \nabla^2 F(v^*) s = 
2 \sum_{i \in \mathcal{A}} s_i^2 [\nabla f(x^*)]_i + 4 \sum_{i\in \mathcal{I} } \sum_{j \in \mathcal{I}} [\nabla^2 f(x^*)]_{ij}  (v^*_is_i) (v^*_js_j).
\]
Since $\nabla f(x^*)\geq 0$ and since $w^\top \nabla^2 f(x^*) w = \sum_{i,j\in \mathcal{I}} \nabla^2 f(x^*)w_iw_j\geq 0 $ for any $w$ with $w_i =0$ for $i \in \mathcal{A}$, both summation terms are nonnegative. 
If $s_i \neq 0$ for any $i \in \mathcal{I}$, the second term is strictly positive, by 2S conditions for \eqref{eq:bc}. 
If $s_i \neq 0$ for any $i \in \mathcal{A}$, we have by strict complementarity that $[\nabla f(x^*)]_i>0$, so term $i$ in the first summation is strictly positive. 
Thus, $s^T \nabla^2 F(v^*) s>0$ for all $s \ne 0$, proving that $v^*$ is a 2S point for \eqref{eq:bc.ssv}.

Suppose now that 2S conditions hold for \eqref{eq:bc.ssv} at $v^*$, and let $x^* = v^* \odot v^*$. 
From Theorem~\ref{th:bc}, we know $x^*$ satisfies 2N coniditions and thus 1N conditions.
By 2S conditions for \eqref{eq:bc.ssv}, we have $s^T \nabla^2 F(x^*) s  >0$ for all $s \neq 0$. 
Choosing any $i \in \mathcal{A}$,  we have $v^*_i=0$. 
For this $i$, set $s_i=1$, with $s_j=0$ for all $j \neq i$. 
From \eqref{eq: q_form_F}, we have $0<s^T \nabla^2 F(x^*) s = 2 [\nabla f(x^*)]_i$, verifying that strict complementarity holds. 
Now choosing any $s \neq 0$ such that $s_i=0$ for all $i \in \mathcal{A}$.
For any such $s$, from strict complementarity, we know 
$\sum_{i=1}^n s_i^2 [\nabla f(x^*)]_i = \sum_{i\in \mathcal{A}} s_i^2 [\nabla f(x^*)]_i =0$. Hence, 
from \eqref{eq: q_form_F}, we have for any such $s$,  
$
0<s^T \nabla^2 F(x^*) s = 4 \sum_{i,j \in \mathcal{I}} (v^*_is_i) (v^*_js_j) [\nabla^2 f(x^*)]_{ij},
$
which implies that 
$w^\top \nabla^2 f(x^*)w = 
\sum_{i,j\in \mathcal{I}} \nabla ^2 f(x^*)w_iw_j>0$ for all $w\not =0$ with $w_i=0$ for $i\in \mathcal{A}$, as $v_i^*,v_j^*>0$ for $i,j\in \mathcal{I}$. This fact, together with strict complementarity, implies that 2S conditions are satisfied for \eqref{eq:bc} at $x^*$, as claimed.
\end{proof}
}

\rev{
\subsection{Application: Finding a 2N point of \eqref{eq:bc} via gradient descent for \eqref{eq:bc.ssv}} \label{sec: 2NGD}
In this section, we demonstrate an algorithmic implication of Theorem \ref{th:bc}: For a wide class of problems, one can find a weak 2N point of \eqref{eq:bc} via ``vanilla" gradient descent (GD) for \eqref{eq:bc.ssv} from almost any initial point. 
On the other hand, finding a 2N point of \eqref{eq:bc} directly requires more sophisticated algorithms. 
In particular, it is not known whether a projected gradient algorithm for \eqref{eq:bc} can find a 2N point with high probability from a random initialization. 

From \cite{lee2016gradient,panageas2016gradient,shub2013global}, for unconstrained minimization of a smooth function $h:\mathbb{R}^n \to \mathbb{R}$, it is known that GD with almost any starting point will {\em not} converge to a point that satisfies 1N conditions but not 2N conditions when the Hessian of the objective $h$ is bounded, that is, there is $L>0$ such that $ \| \nabla ^2 h(x)\|_2 \leq L$ for any $x$.
That is, GD will either converge to a 2NP or diverge for almost all initializations. 
However, even if the objective $f$ in \eqref{eq:bc} has a bounded Hessian, it is not true in general that $F$ in \eqref{eq:bc.ssv} has a bounded Hessian.
(Consider $f(x) = x^2$ and $F(v) = v^4$.) 
Thus, we cannot apply this result directly to the reformulation \eqref{eq:bc.ssv}.
To circumvent this issue, we consider objective functions $h$ with compact sublevel sets. 
We also require $h$ to be semialgebraic \cite[Definition~5.1]{bolte2014proximal}, meaning that its graph is a finite union and intersection of semialgebraic sets.\footnote{For a function $H:\mathbb{R}^{n_1}\rightarrow \mathbb{R}^{n_2}$, its graph is $\{(x,H(x))\mid x\in \mathbb{R}^{n_1}\}$. A set $A$ in $\mathbb{R}^n$ is semialgebraic if $A =\{x\mid p_i(x)=0,\;q_j(x)<0,i=1,\dots,I, j= 1,\dots,J\}$ for some polynomials $p_i$ and $q_j$ and some integers $I,J>0$.} 
We introduce this property because (i) it is satisfied for a wide range of applications, including  matrix optimization \cite{absil2008optimization,edelman1998geometry} and modern neural networks with ReLU activation \cite{davis2020stochastic}, and (ii) it ensures the convergence of gradient descent for $h$ to a 1P if the iterates are bounded \cite[Theorem~1, 
Theorem~3]{bolte2014proximal}.\footnote{More generally, the convergence holds for functions definable in an $o$-minimal structure over $\mathbb{R}$ \cite{attouch2010proximal,van1998tame}. The results in this section continue to hold for definable functions.} 

\begin{definition} \label{def:semi}
    $\mathcal{F}_s$ is the set of all functions $f:\mathbb{R}^n\rightarrow \mathbb{R}$ such that (i) $f$ is semialgebraic, and (ii) for any $\alpha\in \mathbb{R}$, the sublevel set $\{x\mid f(x)\leq \alpha\}$ is compact. 
\end{definition}

Our first theorem shows that for any $h\in \mathcal{F}_s$, gradient descent from an initial point $x_0$ (with constant stepsize bounded by a positive quantity that may depend on  $x_0$) converges to a 2N point of $\min \, h(x)$ for almost all choices of  $x_0$.
\begin{theorem}\label{thm: gdcompactsub}
    Suppose $h: \mathbb{R}^n \rightarrow \mathbb{R}$ is second-order continuously differentiable and $h\in \mathcal{F}_s$. Then for almost any initialization $x_0$, there is a step size $\eta>0$ such that gradient descent (GD) for $h(x)$ with stepsize $\eta$ starting at $x_0$ will converge to a 2N point $x_s$ (that is, $\nabla h(x_s)=0$ and $\nabla^2 h(x_s)\succeq 0$).
\end{theorem}
\begin{proof}
For any $\alpha \in \mathbb{Z}$, define the sublevel set $B_\alpha := \{x\mid h(x)\leq \alpha\}$ and note that it is compact by \cref{def:semi}. 
Define $A_\alpha \subset B_{\alpha}$ to be the set of points $x \in B_\alpha$ such that there is a stepsize $\eta>0$ such that GD starting at $x$ with stepsize $\eta$ converges to a 2N point. 
We show in the remainder of the proof that the following claim holds: 

\begin{equation} \label{eq: GD_claim}
\parbox{\dimexpr\linewidth-4em}{%
For any $\alpha \in \mathbb{Z}$ and almost any $x_0 \in B_\alpha$, there is a stepsize $\eta>0$ such that 
    GD starting at $x_0$ with stepsize $\eta$ converges 
to a 2N point, or equivalently, for any $\alpha \in\mathbb{Z}$, the Lebesgue measure of $B_\alpha \setminus A_\alpha$ is $0$.}
\end{equation}
 Given this claim, and noting that $\cup_{\alpha \in \mathbb{Z}} B_\alpha = \mathbb{R}^n$ due to $h\in \mathcal{F}_s$, the set $\cup_{\alpha\in \mathbb{Z}} A_\alpha$ contains all starting points for which that GD (with a step length appropriate to that starting point) converges to a 2N point. 
 The claim \eqref{eq: GD_claim} implies that the complement of $\cup_{\alpha\in \mathbb{Z}} A_\alpha$ has measure $0$, which proves our theorem.

To complete the proof, we verify the claim \eqref{eq: GD_claim}. Given $\alpha \in \mathbb{Z}$, we introduce 
\[
L_{\alpha} := \max_{x\in B_\alpha} \twonorm{\nabla^2 h(x)}, \quad r_\alpha = \max_{x\in B_\alpha}\twonorm{x} + \frac{1}{L_{\alpha}}\twonorm{\nabla h(x)},
\]
and note that both quantities are finite, by compactness of $B_{\alpha}$.
Denoting by $B(r)$ the closed ball of radius $r$ centered at $0$, we define $L_{r_\alpha} := \max_{x\in B(r_\alpha)} \twonorm{\nabla^2 h(x)}$. 
Note that $B_\alpha \subset B(r_\alpha)$ and $L_{r_\alpha} \ge L_{\alpha}$.
Next, we consider following three claims, showing that they collectively imply  \eqref{eq: GD_claim}. 

\begin{enumerate}
    \item[(i)] For any $x_0 \in B_\alpha$, GD for $h$ with any stepsize in $(0,{1}/{L_{r_\alpha}})$ has a trajectory that stays inside in $B_\alpha$ and converges to a 1P for $h$.
    \item[(ii)] There is a smooth function $g_{\alpha}$ with a bounded Hessian such that $g_\alpha$ and $h$ are identical in $B(r_\alpha+1)$.
    \item[(iii)] For any smooth function $g$ with domain $\mathbb{R}^n$ and bounded Hessian, defining  $\hat{L}_{g} = \max_{x\in \mathbb{R}^n} \| \nabla ^2 g(x)\|_2$, for almost all $x_0\in\mathbb{R}^n$, GD for $g$ initialized at $x_0$ with stepsize $\eta  \in (0,{1}/{\hat{L}_g})$  will {\em not} converge to a point that is a 1P but not a 2NP of $g$. 
\end{enumerate}
Item (ii) implies that $g_{\alpha}$ and $h$ have identical 1P and 2N points within the set $B(r_\alpha)$, while item (i) and (ii) imply that GD initialized in $B_\alpha$ with any stepsize in $(0,1/L_{r_\alpha})$ will produce the same set of iterates for $g_{\alpha}$ and $h$.
Note that if we set $g = g_{\alpha}$ in (iii) and note that $\nabla^2 g_{\alpha}(x) = \nabla^2 h(x)$ for $x \in B(r_{\alpha}) \supset B_{\alpha}$, we have $\hat{L}_{g_{\alpha}} \ge L_{r_{\alpha}}$.
Hence, we see that by choosing a stepsize $\eta \in 
(0, 1/\hat{L}_{g_{\alpha}})$, GD initialized with almost any $x_0 \in B_\alpha$ converges to a 1P of $h$ (due to (i)), and this 1P is also a 2NP of $h$ (due to (i), (ii), and (iii)). 
Recall that the set $A_{\alpha} \subset B_{\alpha}$ contains all such points $x_0$. Hence, we have because of the ``almost any" property that $B_\alpha \setminus A_\alpha$ has Lebesgue measure zero.
This is exactly our claim \eqref{eq: GD_claim}. 
Since (iii) follows from \cite[Theorem 2]{panageas2016gradient}, we are left to show (i) and (ii).

For item (i), note that for any $x \in B_\alpha$, we have $x - \eta \nabla h(x) \in B(r_\alpha)$ provided that $\eta \leq {1}/{L_{\alpha}}$. Hence, for any $\eta \in (0, {1}/{L_{r_\alpha}})$, from  
\[
h(x-\eta \nabla h(x))\leq 
    h(x) -\eta \twonorm{\nabla h(x)}^2 + \frac{L_{r_\alpha}\eta^2}{2} \twonorm{\nabla h(x)}^2\leq h(x),
\]
we know with an initial point $x_0 \in B_\alpha$, the trajectory of GD for $h$ with a stepsize smaller than ${1}/{L_{r_\alpha}}$ always lies in $B_\alpha$.  To show that GD converges, from the assumption that $h$ is semialgebraic, we know that the Kurdyka-{\L}ojasiewicz inequality holds for $h$ \cite[Theorem~3]{bolte2014proximal}. 
The iterates are  bounded because they are all in $B_\alpha$. 
Hence, from \cite[Theorem~1 and Remark~3 (iv)]{bolte2014proximal}, gradient descent for $h$ with any 
initial $x_0\in B_\alpha$ converges to a 1P for any stepsize in $\eta \in (0,{1}/{L_{r_\alpha}})$.

For item (ii), by a standard construction in analysis, there is a smooth nonnegative function $\phi:\mathbb{R}^n\rightarrow \mathbb{R}$ such that $\phi$ is $1$ for $x\in B(r_\alpha+1)$, $0$ for $x\in B(r_\alpha+2)$. 
The function $g = \phi h$ satisfies the requirement.
\end{proof}

With Theorem \ref{thm: gdcompactsub}, we reach the main conclusion of this section. 
\begin{corollary}
    If the objective $f$ in \eqref{eq:bc} belongs to $\mathcal{F}_s$, then for almost all initializations $v_0$, gradient descent with a small enough constant stepsize for $F$ in \eqref{eq:bc.ssv} converges a point $v_s$ such that $x_s = v_s \odot v_s$ is a 2N point of \eqref{eq:bc}.
\end{corollary}
\begin{proof}
    With Theorem \ref{th:bc} and \ref{thm: gdcompactsub}, we need only to show that the function $F$ defined in \eqref{eq:bc.ssv} also belongs to $\mathcal{F}_s$. 
    The compact sublevel set property follows from the following argument. 
    For any $\alpha\in \mathbb{R}$,  $f(x)\leq \alpha$ implies that $|x_i| \le b_{\alpha}$ for some $b_{\alpha}>0$ and all $i=1,2,\dotsc,n$. Thus, $|v_i|\leq \sqrt{b_{\alpha}}$ for all $i$ when $x_i = v_i^2$. 
    Hence, the sublevel set of $F$ is also bounded. 
    The semialgebraic property for $F$ follows from the property that composition of semialgebraic functions is semialgebraic \cite[Example~5.1]{bolte2014proximal} and the fact that $v \mapsto v\odot v$ is semialgebraic because its graph is semialgebraic.  
\end{proof}
}

\subsection{Application: $\ell_1$ regularization}
\label{sec: app_L1}

\rev{
The $\ell_1$ norm is used in optimization formulations as a regularizer to encourage sparsity, often by adding  a multiple of $\| x \|_1$ to the original smooth objective function $h:\RR^{n} \rightarrow \RR$. 
Consider the regularized  optimization problem
\begin{equation}\label{eq: L1} 
    \min_{x \in \mathbb{R}^n} \, h(x) + \lambda \onenorm{x},
\end{equation}
where $\lambda>0$ is a tuning parameter. 
It is well-known the following constrained formulation is equivalent to the original formulation in terms of global minimizers:
\begin{equation}\label{eq: L1bc} 
\min_{\bar{x} = (x_+,x_-) \in \mathbb{R}^{2n}} \, \bar{h}(\bar{x}) := h(x_+-x_-) + \lambda \left( \mathbf{1}_n^\top x_+ + \mathbf{1}_n^\top x_-\right)\;\; 
    \mbox{s.t.} \;\; \bar{x}\geq 0.
\end{equation}
This problem has the form of \eqref{eq:bc}.
From Lemma \ref{lem: 1PL1L1bc}, this equivalence actually continues to hold for 1P (points satisfying first-order necessary conditions). 
Consider now the direct substitution $\bar{x} = v\odot v$ for 
$
v :=  \begin{bmatrix}
    v_+ \\ 
    v_-
\end{bmatrix} \in \RR^{2n}.
$
We can reformulate \eqref{eq: L1bc}  in the manner of \eqref{eq:bc.ssv} as
\begin{equation}\label{eq: L1bc.ssv} 
     \min_{v_+,v_-\in \RR^{n}} \,  h(v_+\odot v_+ - v_-\odot v_- ) + {\lambda}\big(\twonorm{v_+}^2+\twonorm{v_-}^2\big).
\end{equation}
Theorem~\ref{th:bc} can be applied to show a direct correspondence between 2N points for \eqref{eq: L1bc} and \eqref{eq: L1bc.ssv}. 
When $h$ is convex, a 2NP (indeed, even a 1P) of \eqref{eq: L1bc} is a global optimum.
Thus, our result shows that we can find the global optimum of  \eqref{eq: L1bc} (as well as the global optimum of \eqref{eq: L1}) by finding a 2NP of \eqref{eq: L1bc.ssv} if $h$ is convex. 
Theorem~\ref{thm: gdcompactsub} suggests that such points could be found by applying gradient descent to \eqref{eq: L1bc.ssv} from a random starting point, provided that $h$ is also semialgebraic.
For general smooth $h$, we have from 
Theorem~\ref{th:bc} and Theorem \ref{thm: gdcompactsub} that  we can find a 
2NP of \eqref{eq: L1bc} 
by finding a 2NP of \eqref{eq: L1bc.ssv}, by using gradient descent (GD) if $h$ is semialgebraic. 
Furthermore, due to equivalence of 1Ps between \eqref{eq: L1} and \eqref{eq: L1bc} (from Lemma \ref{lem: 1PL1L1bc}), the 2NP  found by GD   also gives a 1P of \eqref{eq: L1} . 

In Section \ref{sec: cls}, we investigate numerically a constrained form of \eqref{eq: L1} where $h$ is a sum-of-squares function. We also extend the squared-variable  approach to tackle pseudo-norms. 
}

\section{Properties of the SSV reformulation} \label{sec:ssv.theory}

In this section, we examine properties of the formulation \eqref{eq:ssv} and its relation to the inequality constrained optimization problem \eqref{eq:f}, focusing on second-order necessary (2N) optimality conditions and approximate versions of these conditions. 
We begin in \Cref{sec:2on} with some definitions and discussion of the exact form of the 2N conditions.
\Cref{sec:app2N} \rev{discusses} approximate forms of these conditions for both \eqref{eq:f} and \eqref{eq:ssv}, and the relationships between them.
\Cref{sec:cq} discusses constraint qualifications for the two formulations, while \Cref{sec:complex} discusses \rev{a complexity result} for solving \eqref{eq:f} by applying an algorithm for equality constrained optimization to \eqref{eq:ssv}.

\subsection{Second-order necessary (2N) conditions}
\label{sec:2on}

The Lagrangians of the formulations \eqref{eq:f} and \eqref{eq:ssv}  are 
\begin{align}\label{eq:Lo_and_L}
\cL_o(x,s) = f(x) - \sum_{i=1}^m s_i c_i(x),   \quad \text{and} \quad 
\cL(x,v,s) = f(x) - \sum_{i=1}^m s_i (c_i(x)-v_i^2),
\end{align}
respectively.  The second-order necessary (2N) conditions for \eqref{eq:f} can be defined in terms of the ``original" Lagrangian, which we denote by  $\cL_o$.
\begin{definition}[2N conditions for \eqref{eq:f}]\label{def: sospOriginalProblem} 
A point $(x,s) \in \R^n \times \R^m$ satisfies {\em second-order necessary (2N) conditions} for \eqref{eq:f} if it satisfies the first-order necessary (KKT) conditions along with an additional condition (stated last) involving $\nabla^2_{xx} \cL_o(x,s)$. 
The KKT conditions are:
\begin{subequations}
\label{eq: originalG.kkt}
\begin{align}
    & \nabla_x \cL_o (x,s) = \nabla f(x) - \sum_{i=1}^ m s_i \nabla c_i(x) = 0,  \\
    & s \geq 0, \quad c(x) \geq 0,\\ 
    & s_i c_i(x) = 0, \quad i=1,2,\dotsc,m.
\end{align}
\end{subequations}
Given $(x,s)$, we define the index sets $\cI$, $\cA$, and $\cD$ to be a partition of $\{1,2,\dotsc,m\}$  satisfying
\begin{subequations}
\label{eq:Gbc.kkt}
\begin{align}
    i \in \cI & \Rightarrow \; c_i(x)>0, \; s_i =0, \\
    i \in \cA & \Rightarrow \; c_i(x) = 0 , \; s_i >0, \\
    i \in \cD & \Rightarrow \; c_i(x) =0, \; s_i =0.
\end{align}
\end{subequations}
To complete the definition of 2N conditions, we need that for all $w$ with $\nabla c_i(x)^\top w=0$ for $i \in \cA$ and $\nabla c_i(x)^\top  w  \ge 0$ for $i \in \cD$, we have
\begin{equation} \label{eq:2o.strong}
w^T \nabla^2 _{xx} \cL_o(x,s ) w \ge 0.  
\end{equation}

We say that $(x,s)$ satisfies {\em weak second-order necessary (weak 2N) conditions} if it satisfies the KKT condition \eqref{eq: originalG.kkt} and the following (less restrictive) condition on $\nabla^2_{xx} \cL_o(x,s)$:
\begin{equation} \label{eq:2o.weak}
w^T \nabla^2_{xx} \cL_o (x,s) w \ge 0, \quad \mbox{for all $w$ with $\nabla c_i(x)^\top w=0$ for $i \in \cA \cup \cD$.}    
\end{equation}
\end{definition}

Next we define 2N conditions for \eqref{eq:ssv}.
\begin{definition}[2N conditions for \eqref{eq:ssv}] \label{def: sospSSV}
A point $(x,v,s) \in \R^n \times \R^m \times \R^m$ satisfies {\em second-order necessary (2N) conditions} for \eqref{eq:ssv} if it satisfies the following KKT condition:
\begin{subequations}
\label{eq: Gbc.ssv.kkt}
\begin{align}
    \nabla_x \cL (x,v,s) &= \rev{\nabla_x \cL_o (x,v,s) = }\nabla f(x) - \sum_{i=1}^ m s_i \nabla c_i(x) = 0,  \\
 \nabla_v \cL (x,v,s) &= 2 s \odot v= 0,  \\
     c_i ( x) &= v_i^2, \quad i=1,2,\dotsc,m,
\end{align}
\end{subequations}
along with the condition that for all $(w,z)\in \real^{n + m}$ satisfying $\nabla c_i(x)^\top w = 2v_i z_i$ for $i=1,2,\dotsc,m$, we have
\begin{equation} \label{eq:Gbc.ssv.soc}
w^T \nabla^2_{xx}\cL (x,v,s ) w  + 2\sum_{i=1}^m s_i z_i^2 
\rev{= w^T \nabla^2_{xx}\cL_o (x,v,s ) w  + 2\sum_{i=1}^m s_i z_i^2 }
\ge 0.
\end{equation}
\end{definition}

The following result relates 2N points for \eqref{eq:ssv} to weak 2N points for \eqref{eq:f}.
\begin{theorem}\label{thm: SOCequavalence}
If  $(x,v,s)$ is a 2N point of \eqref{eq:ssv} then $(x,s)$ is a weak 2N point for \eqref{eq:f}. 
Conversely, if $(x,s)$ is a weak 2N point for \eqref{eq:f}, then any point $(x,v,s)$ with $v_i ^2 = c_i(x)$ is a 2N point for \eqref{eq:ssv}. 
\end{theorem}
\begin{proof}
The first claim is essentially shown in Bertsekas~\cite[Section~3.3.2]{Ber99}. (We need only observe that the requirement in that proof that $v_i \ge 0$ is unnecessary.) 
%
Here, we provide proof of the first claim to make the paper self-contained.

Let us prove the direction that if $(x,v,s)$ is a 2N point of \eqref{eq:ssv}, then $(x,s)$ is a weak 2N point for \eqref{eq:f}.

We adopt the notation of Definitions \ref{def: sospOriginalProblem} and \ref{def: sospSSV}. Note that the last two conditions in  \eqref{eq: Gbc.ssv.kkt} imply that $c(x)\geq 0$ and $s _i c_i(x) = 0$ for all $i$. Hence,
\begin{equation} \label{eq: Gbc.ssv.kkt.a}
    c_i(x) >0 \Leftrightarrow v_i \neq 0 \Rightarrow s_i=0, \;\forall \,i \in \cI.
\end{equation}
However, the KKT condition itself does not guarantee $s_i \geq 0$ for $i \in \cA \cup \cD$.

Due to \eqref{eq: Gbc.ssv.kkt.a}, the condition \eqref{eq:Gbc.ssv.soc} reduces to the requirement that for all $(w,z)\in \real^{n + m}$ with $\nabla c_i(x)^\top w = 0$ for $i\in \cA \cup \cD$ and $\nabla c_i(x)^\top w = 2v_i z_i$ for $i\in \cI$, we have
\begin{equation}\label{eq:Gbc.socReduced}
 w^T \nabla^2_{xx}\cL (x,v,s ) w  + 2\sum_{i \in \cA \cup \cD} s_i z_i^2 \ge 0.
\end{equation}
Now for each $i \in \cA \cup \cD$, we set $w=0$ and $z_i = 1$, and set the other entries of $z$ to $0$. 
We thus obtain from \eqref{eq:Gbc.socReduced} that $s_i \geq 0$ for each $i \in \cA \cup \cD$. 

Note that $\nabla_x \cL(x,v,s) = \nabla _x \cL_o(x,s)$ and $\nabla_{xx} \cL(x,v,s) = \nabla _{xx} \cL_o(x,s)$. To prove the second order condition \eqref{eq:2o.weak}, for any $w$ for which $\nabla c_i(x)^Tw=0$ for all $i \in \cA \cup \cD$, we set $z_i = \nabla c_i(x)^\top w/(2v_i)$ for $i\in \cI$ and $z_i = 0$ for $i\in \cA \cup \cD$ in \eqref{eq:Gbc.socReduced}. Such a choice of $z_i$ in \eqref{eq:Gbc.socReduced} gives us 
\begin{equation}\label{eq:Gbc.socReducedTwo}
\begin{aligned} 
    &w^T \nabla^2_{xx}\cL_o (x,s ) w
    = w^T \nabla^2_{xx}\cL (x,v,s ) w  \ge 0, \\
    & \mbox{for all $w\in \real^{n}$ with $\nabla c_i(x)^\top w = 0$ for $i\in \cA \cup \cD$.} 
\end{aligned}
\end{equation}
Thus we have shown that a 2N point for \eqref{eq:ssv} is a weak 2N point for \eqref{eq:f}.

For the reverse implication, we let $(x,s)$ be a 2N point for \eqref{eq:f}.
The first-order conditions \eqref{eq: Gbc.ssv.kkt} for \eqref{eq:ssv} are immediate from the KKT conditions \eqref{eq: originalG.kkt} for \eqref{eq:f}. 
Since $s_i \ge 0$, $i=1,2,\dotsc,m$, the second (summation) term in \eqref{eq:Gbc.ssv.soc} is always nonnegative.
For the first term, we note that any $(w,z)$ satisfying $\nabla c_i(x)^\top w = 2v_i z_i$ for $i=1,2,\dotsc,m$ actually satisfies $\nabla c_i(x)^Tw=0$ for $i \in \cA \cup \cD$, since $v_i=0$ for such $i$. 
Therefore $w^T \nabla_{xx}^2 \cL(x,v,s) w = w^T \nabla_{xx}^2 \cL_o(x,s) w \ge 0$ for such $w$, from \eqref{eq:2o.strong} and 
\editContent{$\nabla_{xx}^2 \cL(x,v,s) =  \nabla_{xx}^2 \cL_o(x,s)$.}
Thus the first term in \eqref{eq:Gbc.ssv.soc} is also nonnegative for all $(w,z)$ satisfying $\nabla c_i(x)^\top w = 2v_i z_i$ for $i=1,2,\dotsc,m$.
Thus $(x,s,v)$ is a 2N point for \eqref{eq:ssv}, as claimed.
\end{proof}

Note that the index set $\cD$ is empty if strict complementarity holds for the original problem \eqref{eq:f}. 

\subsection{Approximate 2N conditions}
\label{sec:app2N}
In this section, we define approximate forms of the 2N conditions for both \eqref{eq:f} and \eqref{eq:ssv}, and prove that approximate 2N points of \eqref{eq:ssv} are approximate 2N points of \eqref{eq:f}.

\begin{definition}[Approximate 2N point for \eqref{eq:f}] \label{def:opt.f}
Define the $\alphac$-active set for a point $x$: 
\begin{equation} \label{eq:defAc}
\cA_{\alphac} = \{i\mid c_i(x)\leq \alphac\}.
\end{equation} 
A point $(x,s)$ satisfies
{\em  $(\epsfoc,\epspf,\epscs,\epspd,\epssoc,\alphac)$-approximate 2N conditions} for \\
\eqref{eq:f} if it satisfies the following approximate KKT conditions:
\begin{subequations} \label{eq: Gbc.kkt.approx}
\begin{align}
\label{eq: Gbc.kkt.approx.A}
    \twonorm{
    \nabla_x \cL_o (x,s)
    }&\leq \epsfoc  & \hfill \text{(first order condition)},\\
    \label{eq: Gbc.kkt.approx.B}
    c_i(x)&\geq -\epspf \text{  for all $i$}& \hfill \text{(primal feasibility)},\\
    \label{eq: Gbc.kkt.approx.C}
    \twonorm{s \odot c(x)}& \leq \epscs & \hfill \text{(complementary slackness)},\\
    \label{eq: Gbc.kkt.approx.D}
    s_i + \rev{a_i  }c_i(x)& \geq -\epspd \text{  for all $i$}& \hfill \text{(primal-dual feasibility)},
\end{align}
\end{subequations}
\rev{for some $a_i \ge 0, i=1,\dots, n$ independent of $(\epsfoc,\epspf,\epscs,\epspd,\epssoc)$ (but possibly depending on $x$ and $s$),} 
in addition to the following condition: For all $w \in \R^n$ such that  $\nabla c_i(x)^\top  w=0$ for all $i\in \cA_{\alphac}$, the following inequality holds:
\begin{equation} \label{eq:Gbc.soc.approx}
w^T \nabla^2_{xx}\cL_o (x,s ) w 
\ge -\epssoc \twonorm{w}^2.
\end{equation}
\end{definition}

If the parameters $\epsfoc,\epspf,\epscs,\epspd,\epssoc,\alphac$ all equal to $0$, our definition reduces to the weak 2N conditions. 
\rev{
If the Jacobian of the active constraints, $J_{\cA_{\alphac}}(x) = [\nabla c_i^\top (x)]_{i \in \cA_{\alphac}}$, has full row rank,  we show in Theorem~\ref{thm: approximateSOC} that an approximate 2N point of \eqref{eq:ssv} (Definition~\ref{def:ssv.approx}) gives an approximate 2N point of \eqref{eq:f} (Definition~\ref{def:opt.f}). 
Under the condition that $J_{\cA_{\alphac}}(x)$ has full row rank, an explicit formula for $a_i$, $i \in \cA_{\alphac}$ is given in \eqref{eq: approx_soc_notations}.
(For $i \not \in \cA_{\alphac}$, we can set $a_i=0$, as shown in the proof of Theorem \ref{thm: approximateSOC}.)
}


\begin{definition}[Approximate 2N point for \eqref{eq:ssv}]
\label{def:ssv.approx}
A point $(x,v,s)$ satisfies {\em $(\epsilon_1,\epsilon_2,\epsilon_3)$-approximate 2N conditions for \eqref{eq:ssv}} if it satisfies the following approximate KKT condition:
\begin{subequations} \label{eq: Gbc.ssv.kkt.approx}
\begin{align}
\label{eq: Gbc.ssv.kkt.approx.A}
    \twonorm{
    \begin{bmatrix}
    \nabla_x \cL (x,v,s)\\
    \nabla_v \cL (x,v,s))
    \end{bmatrix}
    } =
    \twonorm{
    \begin{bmatrix}
    \nabla f(x) - \sum_{i=1}^ m s_i \nabla c_i(x)\\
   2 s \odot v 
    \end{bmatrix}
    }&\leq \epsilon_1\\
    \label{eq: Gbc.ssv.kkt.approx.B}
    \twonorm{c(x)- v\odot v}&\leq \epsilon_2,
\end{align}
\end{subequations}
along with the following second order condition: 
For all $(w,z)\in \R^n \times \R^m$ with $\nabla c_i(x)^\top w = 2v_i z_i$, $i=1,2,\dotsc,m$, we have
\begin{equation} \label{eq:Gbc.ssv.soc.approx}
w^T \rev{\nabla^2_{xx}\cL_o} (x,v,s ) w  + 2\sum_{i=1}^m s_i z_i^2 
\ge -\epsilon_3 \twonorm{(w,z)}^2.
\end{equation}
\end{definition}

We are ready to state the theorem that relates approximate 2N points of \eqref{eq:f} and  \eqref{eq:ssv} to each other.
\begin{theorem}\label{thm: approximateSOC}
Suppose that $(x,v,s)$ is an $(\epsilon_1,\epsilon_2,\epsilon_3)$-approximate 2N point for \eqref{eq:ssv} \rev{with $\epsilon_i \in (0,1]$, $i=1, 2, 3$.}
Choose $\alphac \ge 2 \epsilon_2$ to define an approximately active set $\cA_{\alphac}$ from \eqref{eq:defAc}, and suppose that the active constraint Jacobian \rev{$J_{\cA_{\alphac}}(x) := [\nabla c_i^\top (x)]_{i \in \cA_{\alphac}}$} corresponding to $\cA_{\alphac}$ has full rank.
Then there are positive quantities $\epsfoc$, $\epspf$, $\epscs$, $\epspd$, and $\epssoc$ satisfying
\begin{equation}
\label{eq:bw1}
\begin{aligned}
 &    \epsfoc =\epsilon_1, \quad \epspf = \epsilon_2, \quad \epscs = O(\epsilon_1+\epsilon_2), \\
 &  \epspd = O\Big( \frac{\epsilon_1}{\alphac^{1/2}} + \epsilon_2 + \epsilon_3 \Big), \quad
 \epssoc= O \Big( \epsilon_3 + \frac{\epsilon_3}{\alphac} + \frac{\epsilon_1}{\alphac^{3/2}} \Big),
\end{aligned}\end{equation}
such that $(x,s)$ is an  $(\epsfoc,\epspf,\epscs,\epspd,\epssoc,\alphac)$-approximate 2N point for \eqref{eq:f}. 
Here the notation $O$ omits dependence on the constraint $c(x)$, the Jacobian matrix of the constraints, the Lagrange  multiplier vector $s$, and the Hessian matrix $\nabla ^2_{xx}\mathcal{L}_o(x,s)$.
\end{theorem}

Before proving this result, We state a corollary for the case of  $\epsilon_1 = \epsilon_2 =\epsilon_3=\epsilon$.
\begin{corollary}
Suppose a point $(x,v,s)$ is a $\epsilon_1 = \epsilon_2 =\epsilon_3=\epsilon$-approximate 2N point of \eqref{eq:ssv}, for \rev{$\epsilon \in (0,1]$} and that \rev{$\alphac$ satisfies $\alphac \ge 2 \epsilon$ but otherwise is independent of $\epsilon$} in \cref{thm: approximateSOC}.
\rev{Then if the conditions of \cref{thm: approximateSOC} are satisfied for this choice of $(\epsilon_1,\epsilon_2,\epsilon_3)$, we have  that the conclusions of \cref{thm: approximateSOC} hold with 
$\epsfoc$, $\epspf$, $\epscs$,  $\epspd$, and  $\epssoc$ all $O(\epsilon)$ in \cref{def:opt.f}.}
\end{corollary}

\begin{proof}[Proof of \Cref{thm: approximateSOC}]
We first collect two inequalities on $v$ for future use. 
\rev{Using \eqref{eq: Gbc.ssv.kkt.approx.B} and
$\cA_{\alphac}$ defined in \eqref{eq:defAc}, we have}
\begin{subequations}
\begin{alignat}{2}
&v_i^2\leq \epsilon_2 +\alphac,\quad && \text{for all}\quad i \in \cA_{\alphac},\label{eq: UpperBoundOnActiveV}\\ 
&v_i^2\geq -\epsilon_2 + \alphac,\quad && \text{for all}\quad i \notin \cA_{\alphac}.\label{eq: lowerBoundOnInactiveV}
\end{alignat}
\end{subequations}

\rev{We verify the conditions in \cref{def:opt.f} in turn.}
\paragraph{First order condition \cref{eq: Gbc.kkt.approx.A}} 
\rev{The condition $\twonorm{\nabla_x \cL_o(x,s)}\leq \epsfoc = \epsilon_1$ is immediate from \eqref{eq: Gbc.ssv.kkt.approx.A} when we note that $\nabla_x \cL (x,v,s) = \nabla_x \cL_o(x,s)$.
}

\paragraph{Feasibility of $x$ \cref{eq: Gbc.kkt.approx.B}} 
\rev{Setting $\epspf=\epsilon_2>0$, the claim \cref{eq: Gbc.kkt.approx.B} holds immediately when $c_i(x) \ge 0$.
For $i$ such that $c_i(x)<0$, we have from \eqref{eq: Gbc.ssv.kkt.approx.B} that $c_i(x) \ge c_i(x) - v_i^2 \ge -\epsilon_2 = -\epspf$, so the claim holds in this case as well}

\paragraph{Complementary slackness \cref{eq: Gbc.kkt.approx.C}} 
\rev{From $\twonorm{c(x)-v\odot v}\leq \epsilon_2$,  we have
\begin{equation}\label{eq: norm_v}
\infnorm{v} \leq \max_{1\leq i\leq m }\{ \sqrt{|c_i(x)|+\epsilon_2}\},
\end{equation}
Hence, we have 
\begin{equation}
\begin{aligned} 
\twonorm{s \odot c(x)}
&\leq \twonorm{s \odot v\odot v} + \twonorm{s\odot (c(x) - v\odot v)}\\
& \leq \twonorm{s\odot v}\infnorm{v} + 
\twonorm{c(x) - v\odot v} \infnorm{s}
\\
& \leq \tfrac{1}{2}\epsilon_1\max_{1\leq i\leq m}\sqrt{|c_i(x)|+\epsilon_2}+ \epsilon_2 \infnorm{s},
\end{aligned} 
\end{equation}
where in the final inequality we use \eqref{eq: norm_v}, the inequality 
$2\twonorm{s\odot v}\leq \epsilon_1$ from \eqref{eq: Gbc.ssv.kkt.approx.A}, and \eqref{eq: Gbc.ssv.kkt.approx.B}.
}

\paragraph{Primal-dual feasibility \cref{eq: Gbc.kkt.approx.D}} 
\rev{We consider first the indices $i \not\in \Amap_{\alphac}$. Using 
$|2s_{i}||v_{i}|\leq  \twonorm{2s\odot v} \leq  \epsilon_1$ from \eqref{eq: Gbc.ssv.kkt.approx.A}, together with the inequality \eqref{eq: lowerBoundOnInactiveV}, we know 
\begin{equation}\label{eq: dualfeasinactive}
|s_{i}|\leq \frac{\epsilon_1}{2\sqrt{\alphac-\epsilon_2}} 
\le \frac{\epsilon_1}{\sqrt{2 \alphac}},
\end{equation}
where the final bound follows from $\alphac \ge 2 \epsilon_2 \Rightarrow \alphac-\epsilon_2 \ge \tfrac12 \alphac$.}
%
\rev{Thus, $s_i \ge -\frac{\epsilon_1}{\sqrt{2 \alphac}}$ for all $i \not\in \Amap_{\alphac}$, so \cref{eq: Gbc.kkt.approx.D} is satisfied with $a_i=0$ for a choice of $\epspd$ satisfying \eqref{eq:bw1}.}

\rev{We now examine the indices $i\in \cA_{\alphac}$. 
For any $a_i\geq 0$, from \eqref{eq: Gbc.ssv.kkt.approx.B}, we have  
\begin{equation}\label{eq: pdsv}
    s_i + a_i c_i(x) 
    \geq  s_i + a_i v_i^2 -a_i\epsilon_2.
\end{equation}
Thus, our task is to define $a_i \ge 0$ and find a lower bound on $s_i +a_i v_i^2$. 
We define $J(x) := [\nabla c_j^\top (x)]_{j=1}^m$, the constraint Jacobian of \eqref{eq:f} at $x$. 
Recalling $J_{\cA_{\alphac}}(x) = [\nabla c_j^\top (x)]_{j \in \cA_{\alphac}}$, let
$J_{\cA^c_{\alphac}}(x) = [\nabla c_j^\top (x)]_{j\in \cA_{\alphac}^c}$ be the its complement in $J(x)$. 
We partition the constraint Jacobian matrix $\JSSV(x,v)$ of \eqref{eq:ssv} as follows:
\begin{equation} \label{eq:JSSV}
    J_{\text{SSV}}(x,v) := \begin{bmatrix}
    J_{\cA_{\alphac}}(x) & -2\diag(v_{\cA_{\alphac}}) &  0 \\ 
    J_{\cA^c_{\alphac}}(x) & 0  &  -2\diag(v_{\cA_{\alphac}^c}) 
    \end{bmatrix}.
\end{equation}
Then the approximate second order condition for the SSV problem is that for all $(w,z)$ with $J(x)w = 2\diag (v)z$, there holds the inequality \eqref{eq:Gbc.ssv.soc.approx}. Also, introduce 
the notations\footnote{Here $J_{\cA_{\alphac}}^\dagger = J_{\cA_{\alphac}}^T (J_{\cA_{\alphac}}J_{\cA_{\alphac}}^T)^{-1}$, which is well defined as  $J_{\cA_{\alphac}}(x)$ is full-rank.} 
\begin{equation}\label{eq: approx_soc_notations}
\eta_0 =J_{\cA_{\alphac}}^\dagger e_i,\quad \xi_0 = J_{\cA^c_{\alphac}}\eta_0,\quad \text{and}\quad a_i = 2\max\{0, 
\eta_0 ^\top \nabla ^2 _{xx}\mathcal{L} (x,v,s) \eta_0\}.
\end{equation}
We shall utilize the inequality \eqref{eq:Gbc.ssv.soc.approx} to obtain a lower bound on $s_i + a_i v_i^2$. 
To this end, set $z_{\cA_{\alphac}} = e_i$, 
 $z_{\cA^c_{\alphac}}=  v_i \diag(v_{\cA^c_{\alphac}}^{-1})\xi_0$, 
 and $w =  2v_i\eta_0$.
 It can be verified such $(w,z)$ satisfies 
 $J(x)w = 2\diag (v)z$ using \eqref{eq: approx_soc_notations}.  With this choice of $(w,z)$, the left hand side of \eqref{eq:Gbc.ssv.soc.approx} is 
 \begin{equation}
2s_i + 2\sum_{j\in \cA_{\alphac}^c}s_j z_j^2+ 4v_i^2 \eta_0^T \nabla_{xx}^2 \cL(x,v,s)\eta_0 \leq 2s_i + 2\sum_{j\in \cA_{\alphac}^c}s_j z_j^2 + 2 a_i v_i^2.
 \end{equation}
 Combining the above bound and the condition \eqref{eq:Gbc.ssv.soc.approx} with the choice of $(w,z)$ above,  and with some rearranging, we obtain
 \begin{equation}
\begin{aligned} \label{eq: dualSlackapproximate}
    2s_i + 2 a_i v_i^2
    \ge & -\epsilon_3 (\twonorm{w}^2 +\twonorm{z}^2)
    - 2\sum_{j\in \cA_{\alphac}^c}s_j z_j^2 \\
     \ge & 
     - \epsilon_3 \left(4v_i^2 \twonorm{\eta_0}^2 + 1+ \big\|z_{{\cA^c_{\alphac}}} \big\|_2^2\right)
     - \frac{\sqrt{2}\epsilon_1}{
     \sqrt{\alphac}} \big\|z_{{\cA^c_{\alphac}}}\big\|_2^2,
\end{aligned} 
\end{equation}
where the second inequality uses the bound on $s_j$, $j \in \cA^c_{\alphac}$ from \eqref{eq: dualfeasinactive}. 
From its definition $z_{\cA^c_{\alphac}}= v_i \diag(v_{\cA^c_{\alphac}}^{-1}) \xi_0$, we can bound $\big\|z_{\cA^c_{\alphac}} \big\|_2$ as follows:
 \begin{equation}\label{eq: zinactivebound}
 \big\|z_{\cA_{\alphac}^c}\big\|_2^2
 \le
 \frac{v_i^2}{\min_{j \in \cA^c_{\alphac}} v_j^2}\twonorm{\xi_0}^2 
\le
\frac{\alphac+\epsilon_2}{\alphac-\epsilon_2}  \twonorm{\xi_0}^2 
\le 3 \twonorm{\xi_0}^2,
 \end{equation}
 where the second inequality is due to  the bound on $v$ in \eqref{eq: UpperBoundOnActiveV} and \eqref{eq: lowerBoundOnInactiveV}, while the final inequality follows from $\alphac \ge 2 \epsilon_2$, which implies $\alphac + \epsilon_2 \le  3\alphac/2$ and $\alphac-\epsilon_2 \ge \alphac/2$. We may also bound $v_i^2 \leq \infnorm{c(x)}+ \epsilon_2\leq \infnorm{c(x)}+ 1$ from \eqref{eq: norm_v} and $\epsilon_2 \leq 1$.
Thus, by combining this bound with \eqref{eq: dualSlackapproximate}, \eqref{eq: pdsv}, and \eqref{eq: zinactivebound}, we see the primal-dual inequality \eqref{eq: Gbc.kkt.approx.D} holds with the value of $a_i$ defined in \eqref{eq: approx_soc_notations} and 
for $\epspd$ defined as follows,
 \begin{equation}
\begin{aligned} \label{eq: dualSlackapproximate_two}
    \epspd:=
     \frac{\epsilon_3}{2} \left(4(\infnorm{c(x)}+1) \twonorm{\eta_0}^2 + 1+ 3\twonorm{\xi_0}^2\right)
     + \frac{3\sqrt{2}\epsilon_1}{2
    \sqrt{\alphac}}\max\{\norm{\xi_0}_2^2 ,1\}+a_i \epsilon_2.
\end{aligned} 
\end{equation} 
Note that $\epspd = O({\epsilon_1}/\sqrt{\alphac} +\epsilon_2 + \epsilon_3)$, as claimed
}

\paragraph{Approximate positive semidefiniteness of $\nabla^2_{xx}\mathcal{L}_o(x,s)$ over a subspace of \eqref{eq:defAc}}
\rev{}
For any $w$ s.t. $\nabla c_i(x)^\top  w=0$ for $i\in \cA_{\alphac}$, we 
take $z_i =\frac{ \nabla c_i^\top  w}{2v_i}$ for $i\in \cA^c_{\alphac}$, and set $z_i= 0$ for  $i\in \cA_{\alphac}$. Using this $z$ in the condition \eqref{eq:Gbc.ssv.soc.approx} of \eqref{eq:ssv}, we have 
 \begin{equation}\label{eq: SOCapproxSSVApproxOriginalStep1}
 \begin{aligned} 
w^T \nabla^2_{xx}\cL (x,v,s ) w 
& \ge -\epsilon_3\twonorm{w}^2- 2\sum_{i=1}^m (s_i+\epsilon_3) z_i^2 \\
&= 
-\epsilon_3\twonorm{w}^2- 2\sum_{i \in \cA^c_{\alphac}}(s_i+\epsilon_3)\frac{(\nabla c_i^\top  w)^2}{2v_i^2}.
\end{aligned}
\end{equation}
 \rev{From \eqref{eq: dualfeasinactive}, we have $|s_i|\leq \frac{\epsilon_1}{\sqrt{2\alphac}}$ for any $i \not \in \mathcal{A}_{\alphac}$, while \eqref{eq: lowerBoundOnInactiveV} implies that 
 $v_i^2 \geq\alphac -\epsilon_2 \ge \alphac/2$ for $i\in \cA_{\alphac}^c$}. 
 Thus, by noting that $\nabla^2_{xx}\cL(x,v,s) = \nabla^2_{xx}\cL_0(x,s)$,  we conclude from \eqref{eq: SOCapproxSSVApproxOriginalStep1} that 
\begin{equation}
\begin{aligned} \label{eq: approximateSOC}
 w^T \nabla^2_{xx}\cL_o (x,s ) w \geq -\left(\epsilon_3 + \norm{ J_{\cA_{\alphac}^c}}^2 \left( \frac{2\epsilon_3}{\alphac}+\frac{\sqrt{2} \epsilon_1}{\alphac^{3/2}} \right)\right)\norm{w}^2,
\end{aligned} 
\end{equation}
\rev{verifying our claim about \eqref{eq:Gbc.soc.approx} and completing the proof.}
\end{proof}

\begin{remark}
\rev{If there are equality constraints in Problem \eqref{eq:f}, the conclusion of Theorem \ref{thm: approximateSOC} remains valid, with minor modifications to the proof.
}
\end{remark}

\begin{remark}
It is also possible to add squared-slack variables for just some (not all) inequality constraints. 
The 1N and 2N conditions can be defined accordingly, and the results about (approximate) 2N points in this section and the previous section continue to hold.
\end{remark}

\subsection{Constraint qualifications}
\label{sec:cq}

Our results of the previous sections discuss the relationships between optimality conditions for the original and squared-variable formulations. 
To ensure that points that satisfy these conditions are actually local solutions of the problems in question, we need constraint qualifications to hold.
In this section, we discuss how constraint qualifications (CQ) for the original and squared-variable formulations are related.
Note that in \cref{thm: approximateSOC}, a kind of linear independence constraint qualification (LICQ) was required to show that approximate 2N points for \eqref{eq:ssv} are approximate 2N points for \eqref{eq:f}.

We denoted the constraint Jacobian for  \eqref{eq:ssv} as $\JSSV$, defined as
\begin{equation}\label{eq: JSSV}
    J_{\text{SSV}} (x,v) := \begin{bmatrix}
    J(x) & -2\diag(v)
    \end{bmatrix} \in \real^{m \times (n+m)}
\end{equation}
where $J(x) \in \real^{m\times n}$ is the Jacobian matrix for the vector function $c$ at $x$. 
We denote the set of indices for active constraints of a feasible $x$ as $\mathcal{A}_0 = \{i\mid  c_i(x) =0\}$. The partitioned version of $\JSSV$ in \eqref{eq:JSSV} remains valid with $\mathcal{A}_{\alphac}$ replaced by $\mathcal{A}_0$.

\subsubsection{LICQ and MFCQ for \eqref{eq:f} and \eqref{eq:ssv}}
\rev{Recall that LICQ for \eqref{eq:f} means that $\nabla c_i(x)$, $i\in \mathcal{A}_0$ are linearly independent vectors (or equivalently, $J_{\mathcal{A}_0}(x)$ has full row rank), while LICQ for  \eqref{eq:ssv} means that $\JSSV(x,v)$ has full row rank. 
For $\alphac=0$ that $v_i=0$ for $i \in \cA_0$ and $v_i \neq 0$ for $i \in \cA_0^c$, so  from \eqref{eq:JSSV}, we have 
\[
    J_{\text{SSV}}(x,v) := \begin{bmatrix}
    J_{\cA_0}(x) & 0 &  0 \\ 
    J_{\cA_0^c}(x) & 0  &  -2\diag(v_{\cA_0^c}) 
    \end{bmatrix}.
\]
Since $\diag(v_{\cA_0^c})$ is a nonsingular diagonal matrix, it is clear that $J_{\text{SSV}}(x,v)$ has full row rank if and only if $J_{\cA_0}(x)$ has full row rank. 
Thus the LICQ conditions are identical for \eqref{eq:f} and \eqref{eq:ssv}.
}

The situation is quite different for the Mangasarian-Fromovitz constraint qualification (MFCQ) 
\rev{\cite[Definition 12.5]{nocedal1999numerical}.}
Since all constraints in \eqref{eq:ssv} are equalities, MFCQ is identical to LICQ for this problem.
For \eqref{eq:f}, by constrast, MFCQ is a less demanding condition.

We summarize the relationship between these constraint qualifications as follows:
\begin{multline*}
\text{LICQ for \eqref{eq:f} $\iff$ LICQ for \eqref{eq:ssv} $\iff$ MFCQ for \eqref{eq:ssv}} \\
\text{$\implies$} \text{\rev{MFCQ for \eqref{eq:f}.}}
\end{multline*}

\subsubsection{CQ for \eqref{eq:f} and local minimizer of \eqref{eq:ssv}}
Since we aim to solve \eqref{eq:f} via the reformulation \eqref{eq:ssv}, it is natural to pose CQ on \eqref{eq:f} for local solutions $x$ of \eqref{eq:f} instead of on \eqref{eq:ssv} for its local solutions. Indeed, it is possible that a local solution of \eqref{eq:ssv} does not satisfy any common CQ. Yet, it is still a 2NP because the corresponding local solutions of \eqref{eq:f} is regular enough.

Fortunately, from Theorem \ref{thm: SOCequavalence}, posing CQ on local minimizers of \eqref{eq:f} so that they are weak 2N points of \eqref{eq:f} is enough for local minimizers \eqref{eq:ssv} to be 2N points of \eqref{eq:ssv}. Indeed, suppose $(x,v)$ is a local solution of \eqref{eq:ssv}. We aim to show that $(x,v)$ is a 2N point of \eqref{eq:ssv} based on CQ of \eqref{eq:f} for local solutions. Since $(x,v)$ is a local solution of \eqref{eq:ssv}, the point $x$ is a local solution of \eqref{eq:f}. \footnote{Indeed,
if $(x,v)$ is a local solution of \eqref{eq:ssv}, then $f(x)\leq f(y)$ for all feasible (in the sense of \eqref{eq:ssv}) $(y,w)\in B_x(\epsilon)\times B_v(\epsilon)$ . Here $B_x(\epsilon)$ and $B_v(\epsilon)$ are $\ell_2$ norm balls with centers $x$ and $v$ respectively with a small enough radius $\epsilon$. Thus, $f(x)\leq f(y)$ for all feasible (in the sense of \eqref{eq:f}) $y\in B_x(\epsilon_1)$ for some $\epsilon_1 \leq \epsilon$.} Next, 
from Theorem \ref{thm: SOCequavalence}, if {\em any} CQ guaranteeing 2N conditions is satisfied at a local solution $x$ of  \eqref{eq:f}\footnote{
For example, MFCQ and constant rank condition for \eqref{eq:f}, or the constraints for \eqref{eq:f} are simply linear. These CQs are weaker than or different from LICQ and guarantee that 1N and 2N conditions hold.}, then $(x,v)$ satisfies the 2N conditions for \eqref{eq:ssv}. Moreover, the Lagrangian multipliers are the same for the two problems.\footnote{We also note that any properties of Lagrange multipliers deduced from constraint qualification of \eqref{eq:f} continue to hold in \eqref{eq:ssv}.
}

%
%

\subsubsection{Equality and inequality constraints}
Suppose there are both equality and inequality constraints, denoted by
\begin{equation}\label{eq: constraints}
    c_i(x) = 0, \quad i\in \cE, \quad c_i(x)\geq 0, \quad i \in \cI.  
\end{equation}
The Jacobian of the constraints for the SSV version is (with self-explanatory notation)
\begin{equation}
    J_{\text{SSV}} : \, = \begin{bmatrix}
    J_{\mathcal{\cE}}(x)  & 0 \\ 
    J_{\mathcal{\cI}}(x) & - 2\diag(v)
    \end{bmatrix} \in \RR^{(|\cI|+|\cE|)\times  (n+|\cI|+|\cE|)}.
\end{equation}
If the constraint qualification requires the active constraints to be linearly independent, then $J_{\text{SSV}}$ also has linearly independent rows.

For the special case in which the inequality constraints are all bounds (that is, $c_i(x) = x_i \ge 0$ for $i =1,2,\dotsc,n$ and $\cI = \{1,2,\dotsc,n\}$), we may use direct substitution $x = v\odot v$. In this case, the Jacobian from the direct substitution approach is 
\begin{equation}
    J_{\text{DS}} (v) :\, = 
2  J_{\mathcal{E}}(v\odot v) 
 \diag(v) \in \RR^{|\cE|\times n}.
\end{equation}
Hence LICQ for \eqref{eq:f} again implies LICQ for \eqref{eq:ssv}.

\subsection{Implications for algorithmic complexity}
\label{sec:complex}

The literature on algorithmic complexity for solving equality-constrained problems is vast. 
Here we consider the algorithm ProximalAL \cite[Algorithm 2]{xie2021complexity}, which takes the augmented Lagrangian approach and uses Newton-CG procedure \cite{royer2020newton} as a subroutine. 
The main result \cite[Theorem 4]{xie2021complexity} measures the complexity in terms of the total number of iterations of the Newton-CG procedure, whose main cost for each iteration is \rev{the cost of a matrix-vector product involvng the Hessian (with respect to $(x,v)$) of the function} $\cL(x,v,s) + \rho \frac{\|c(x)-v\odot v\|^2}{2}$ for some $\rho>0$.
By applying this result to the SSV problem \eqref{eq:ssv}, we obtain that within $O(\epsilon^{-7})$ iterations of Newton-CG procedure, ProximalAL outputs an $(O(\epsilon),O(\epsilon),O(\epsilon))$-approximate 2N point  $(x,s, v)$ of  the SSV formulation \eqref{eq:ssv}. 
Hence, by Theorem \ref{thm: approximateSOC}, we know the same $(x,s)$ is an approximate 2N point for  \eqref{eq:f} with approximation measure  $\epsfoc = \epspf = \epscs = \epspd = \epssoc = O(\epsilon)$,  \rev{provided that the full-rank condition on the active constraint Jacobian $J_{\cA_{\alphac}}(x)$ holds}.

\section{Linear programming: SSV and interior-point methods} \label{sec:lpssv}

Here we \rev{describe} a squared-variable reformulation for linear programming in standard form, and discuss solving it with sequential quadratic programming (SQP). 
Steps produced by this approach are closely related to those from a standard primal-dual interior-point method.
Computational performance and comparisons between the two approaches are given in \Cref{sec:computation}.

\subsection{SSV formulation}

Consider the LP in standard form for variables $x \in \real^n$:
\begin{equation}
    \label{eq:lp} \tag{LP}
    \min_x \, c^Tx \quad \mbox{s.t.  $Ax=b$, $x \ge 0$,}
\end{equation}
and suppose this problem is feasible and bounded (and hence has a primal and dual solution, by strong duality).
We introduce variables $v \in \real^n$ to obtain the following nonlinear equality constrained problem:
\begin{equation}
    \label{eq:lp.ssv} \tag{LP-Sq}
    \min_{x,v} \, c^Tx \quad \mbox{s.t. $Ax=b$, $x - v \odot v=0$.}
\end{equation}
The KKT conditions for \eqref{eq:lp.ssv} are that there exist Lagrange multipliers $\lambda \in \real^m$ and $s \in \real^n$ such that
\begin{subequations} \label{eq:lp.kkt}
\begin{align}
\label{eq:lp.kkt.1}
    A^T \lambda + s - c &= 0, \\
    \label{eq:lp.kkt.2}
    Ax - b &=0, \\
    \label{eq:lp.kkt.3}
    x - v \odot v &= 0, \\
    \label{eq:lp.kkt.4}
    2 s \odot v &=0.
\end{align}
\end{subequations}
Consistent with the conventions in interior-point methods, we use the notation
\[
S := \diag (s), \quad V := \diag(v), \quad \text{and}\quad \rev{X:=\diag(x).}
\]
The Lagrangian is
\begin{equation}
    \label{eq:lp.lagr}
    \cL(x,v,\lambda,s) = c^Tx -\lambda^T(Ax-b) - s^T(x-v \odot v).
\end{equation}
The Jacobian of the (primal) constraints is
$
J_{\text{LP-Sq}}:\,=    \left[ \begin{matrix} A & 0  \\ I & -2 V \end{matrix} \right],
$
and its  null space is characterized as follows:
\begin{align}
    \label{eq:cj.null}
   \cN & :=  \nullspace (J_{\text{LP-Sq}}) =
    \left\{ \left[ \begin{matrix} d \\ g \end{matrix} \right]  \, : \, d \in \nullspace (A), \;
    v_i=0 \Rightarrow d_i=0, \; v_i \neq 0 \Rightarrow g_i = \frac{d_i}{2v_i} \right\}.
\end{align}
Note that for indices $i$ for which $v_i=0$, the component $g_i$ of any vector $\left[ \begin{matrix} d \\ g \end{matrix} \right]  \in \cN$ is arbitrary. The Hessian of $\cL$ with respect to its primal variables is
$
    \nabla^2_{(x,v),(x,v)} \cL(x,v) = \left[ \begin{matrix} 0 & 0 \\ 0 & 2 S \end{matrix}  \right].
$ 
Thus, the 2N conditions for \eqref{eq:lp.ssv} require that 
\rev{
\begin{equation}
\label{eq:lp.2o}
\left[ \begin{matrix} d \\ g \end{matrix} \right]^T  \left[ \begin{matrix} 0 & 0 \\ 0 & 2 S \end{matrix}  \right]
\left[ \begin{matrix} d \\ g \end{matrix} \right] \ge 0
\Leftrightarrow 
\sum_{i=1}^n s_i g_i^2 \ge 0, \quad \mbox{for all } \left[ \begin{matrix} d \\ g \end{matrix} \right] \in \cN.  
\end{equation}
For any index $i$ s.t. $v_i =0$, we set $d=0$ and $g =e_i$ in \eqref{eq:cj.null}, 
the conditions  \eqref{eq:lp.2o} then imply that $s_i \ge 0.$ 
}

\begin{theorem} \label{th:lp.ssv} \ 
\begin{itemize}
    \item[(a)] Let $(\bar{x},\bar{v},\bar{\lambda},\bar{s})$ satisfy first-order conditions \eqref{eq:lp.kkt} and second-order necessary conditions \eqref{eq:lp.2o} for \eqref{eq:lp.ssv}. Then $\bar{x}$ is a solution of \eqref{eq:lp} and $(\bar\lambda,\bar{s})$ is a solution of the dual problem 
\begin{equation} \label{eq:lp.dual} \tag{LP-Dual}
    \max_{\lambda,s} \, b^T \lambda \quad \mbox{s.t. } A^T \lambda +s = c, \;\; s \ge 0.
\end{equation}
\item[(b)] Suppose that $\bar{x}$ solves \eqref{eq:lp} and $(\bar\lambda,\bar{s})$ solves \eqref{eq:lp.dual}. Then defining $\bar{v}_i = \pm \sqrt{\bar{x}_i}$, $i=1,2,\dotsc,n$, we have that $(\bar{x},\bar{v},\bar\lambda,\bar{s})$ satisfies second-order conditions to be a solution of \eqref{eq:lp.ssv}.
\end{itemize}
\end{theorem}
\begin{proof}
(a) From \eqref{eq:lp.kkt.3} we have that $\bar{x} \ge 0$. Moreover from  \eqref{eq:lp.kkt.3} and \eqref{eq:lp.kkt.4} we have that $\bar{x}_i >0 \Rightarrow \bar{v}_i \neq 0 \Rightarrow \bar{s}_i=0$. We showed above, invoking second-order conditions, that $\bar{x}_i=0 \Rightarrow \bar{v}_i=0 \Rightarrow \bar{s}_i \ge 0$. Hence, we have shown  that the complementarity conditions $0 \le \bar{x} \perp \bar{s} \ge 0$  hold.
Since in addition we have $A^T \bar\lambda + \bar{s}=c$ and $A \bar{x} = b$ (from \eqref{eq:lp.kkt.1} and \eqref{eq:lp.kkt.2}), it follows that all conditions for $\bar{x}$ to be a primal solution of \eqref{eq:lp} and for $(\bar\lambda,\bar{s})$ to be a dual solution are satisfied.

For (b), note that \eqref{eq:lp.kkt} are satisfied when $\bar{v}$ is defined as given. 
The second-order conditions \eqref{eq:lp.2o} are satisfied trivially, since $\bar{s}_i \ge 0$.
\end{proof}

\subsection{SQP applied to SSV, and relationship to interior-point methods}
\label{sec:sqplp}

Consider SQP applied to the SSV formulation \eqref{eq:lp.ssv}.
The step  $(\Dx,\Dv,\Dlambda,\Ds)$ satisfies
\begin{equation}
    \label{eq:lp.ssv.sqp}
    \begin{aligned}
    \min_{\Dx,\Dv} \, c^T \Dx + & \sum_{i=1}^n s_i \Dv_i^2 \\
    \mbox{s.t.} \quad  A \Dx & = (b-Ax), \quad \text{and}\quad 
 \Dx - 2 v \odot \Dv  = -(x-v\odot v).
    \end{aligned}
\end{equation}
with $\Dlambda$ and $\Ds$ coming from the Lagrange multipliers for the equality constraints in this subproblem.
Equivalently, we can linearize the KKT conditions \eqref{eq:lp.kkt} to obtain the following linear system for the SQP step:
\begin{subequations} \label{eq:lp.sqp}
\begin{align}
\label{eq:lp.sqp.1}
    A^T \Dlambda + \Ds &= r_\lambda := c - A^T\lambda - s, \\
    \label{eq:lp.sqp.2}
    A\Dx  &= r_x := b-Ax, \\
    \label{eq:lp.sqp.3}
    \Dx - 2V \Dv &= r_v := -x + v \odot v, \\
    \label{eq:lp.sqp.4}
    S \Dv + V \Ds &=r_{sv} := -s \odot v.
\end{align}
\end{subequations}
By a sequence of eliminations, we can reduce this linear system to the form $A DA^T \Dlambda = r$ for a certain diagonal matrix $D$ {\em provided} that $s$ has all components positive and $v$ has all components nonzero. 
(In fact, as we discuss below, it is easy to maintain positivity of all components of $s$ and $v$ at every iteration.)
More precisely, we have from \eqref{eq:lp.sqp} that
\begin{equation}\label{eq: lp.sqp.lambda}
    A(S^{-1} V^2)A^\top \Delta \lambda = \frac{1}{2} r_x + A(S^{-1} V^2 )r_\lambda - AS^{-1} Vr_{sv} - \frac{1}{2} Ar_v.
\end{equation}
\rev{Given $\Delta \lambda$, the other quantities $\Delta x$, $\Delta v$, and $\Delta s$ can be obtained from \eqref{eq:lp.sqp}.} 
We refer to the method based on this step as ``SSV-SQP."
%
%


\paragraph{Interior-point methods: PDIP and MPC}
Primal-dual interior-point (PDIP) methods are derived from the optimality conditions for \eqref{eq:lp}, which are 
\begin{subequations} \label{eq:lp.kktip}
\begin{align}
\label{eq:lp.kktip.1}
    A^T \lambda + s - c &= 0, \\
    \label{eq:lp.kktip.2}
    Ax - b &=0, \\
    \label{eq:lp.kktip.3}
    s \odot x &= 0, \quad x \ge 0, \; s \ge 0.
\end{align}
\end{subequations}
(Note the similarity to \eqref{eq:lp.kkt}.)
The search direction for a path-following PDIP method satisfies the following linear system:
\begin{subequations} \label{eq:lp.pdip}
\begin{align}
\label{eq:lp.pdip.1}
    A^T \Dlambda + \Ds &= r_\lambda := c - A^T\lambda - s, \\
    \label{eq:lp.pdip.2}
    A\Dx  &= r_x := b-Ax, \\
    \label{eq:lp.pdip.4}
    S \Dx + X \Ds &=r_{sx} := -s \cdot x + \sigma (s^Tx/n),
\end{align}
\end{subequations}
for some ``centering" parameter $\sigma\in [0,1]$. After a sequence of eliminations, we obtain
\begin{equation}\label{eq: lp.ip.lambda.update}
A(S^{-1} X) A^\top \Delta \lambda =   r_x + A(S^{-1}X)r_\lambda - AS^{-1}r_{sx}.  
\end{equation}

In \Cref{sec:lp_computation}, we compare this SSV-SQP method to the classic Mehrotra predictor-corrector (MPC), which forms the basis of most pract8cal interior-point codes for linear programming. 
This method was described originally in \cite{mehrotra1992implementation}; see also \cite[Chapter~10]{Wri97}. 
Essentially, MPC takes steps of the form \eqref{eq:lp.pdip}, but chooses the centering parameter $\sigma$ {\em adaptively} by first solving for an ``affine-scaling" step (a pure Newton step on the KKT conditions, obtained by setting $\sigma=0$ in \eqref{eq:lp.pdip}) and calculating how much reduction can be obtained in the complementarity gap $s^Tx$ along this direction before the nonnegativity constraints $x \ge 0$, $s \ge 0$ are violated. 
When the affine-scaling step yields a large reduction in $s^Tx$, a more aggressive (smaller) choice of $\sigma$ is made; otherwise $\sigma$ is chosen to be closer to $1$. 

\paragraph{Relationship between SSV-SQP and PDIP} 
The linear systems to be solved for SSV-SQP and PDIP are similar in some respects.
To probe the relationship, we consider first the simple case in which $x = v\odot v$ (that is, the final constraint in \eqref{eq:lp.ssv}  is satisfied exactly), and $\sigma=0$ in \eqref{eq:lp.pdip.4}. 
In the notation of  \eqref{eq:lp.sqp} and \eqref{eq:lp.pdip}, we have  $V r_{sv} = r_{sx}$, $r_v=0$ and $V^2 = X$. 
Thus, the only difference between \eqref{eq: lp.sqp.lambda} and  \eqref{eq: lp.ip.lambda.update} in this case is the coefficient  $\frac{1}{2}$ on $r_x$ in \eqref{eq: lp.sqp.lambda}. 
This difference disappears when $x$ is feasible, that is, $r_x=0$.
Even in this ideal case, the {\em next} iterates of SSV-SQP and PDIP  differ in general. 
Even though $\Delta \lambda$ and $\Delta s$ are the same, the $\Delta x$ step for  PDIP \eqref{eq:lp.pdip}  is 
\begin{equation}\label{eq: lp.pdip.deltax}
    \mbox{PDIP:} \quad \Delta x = S^{-1} r_{sx} -S^{-1}X\Delta s.
\end{equation}
The $\Delta v$ and $\Delta x$ for SQP \eqref{eq:lp.sqp} are 
    \begin{subequations}\label{eq:lp.sqp.deltav&deltax} 
    \begin{align}
        \mbox{SSV-SQP:} \quad \Delta v & = S^{-1} r_{sv} - S^{-1}V\Delta s,\label{eq: eq:lp.sqp.deltav}\\
        \mbox{SSV-SQP:} \quad \Delta x  & = 2 V\Delta v =  2S^{-1}r_{sx} - 2 S^{-1}X \Delta s.\label{eq: eq:lp.sqp.deltax}
    \end{align}
    \end{subequations}
That is, the $\Delta x$ component of the step for SSV-SQP is twice as long as for PDIP. 
Moreover, we no longer have $r_v =0$ for SSV-SQP at the next iteration, since at the new iterates $x^+ := x + \alpha \Delta x$ and $v^+ = v + \alpha \Delta v$, we have
\begin{equation}\begin{aligned}
\mbox{SSV-SQP:} \quad 
v^+ \odot v^+ 
& = 
x + 2\alpha v\odot \Delta v + \alpha ^2 \Delta v \odot \Delta v  \\
& =
x + \alpha \Delta x + \frac{1}{4} \alpha^2 X^{-1} \Delta x \odot \Delta x \\
& =
x^+ + \frac{1}{4} \alpha^2 X^{-1} \Delta x \odot \Delta x 
\end{aligned}\end{equation}
where the second step in this derivation follows from \eqref{eq: eq:lp.sqp.deltax}.
Hence for the new $r_v$ residual, we have
\[
r_v^+ = v^+\odot v^+ - x^+ = \frac{1}{4} \alpha^2 X^{-1} \Delta x \odot \Delta x,
\]
which is in general not $0$. The difference arises from  nonlinearity of $x = v\odot v$.

\paragraph{Stepsize choice} For the stepsize choice in PDIP,
we choose the stepsize to ensure the primal iterate $x$ and dual iterate $s$ are always nonnegative.  We also use different stepsize for primal and dual variables for better practical performance. Precisely, we set the maximum step for primal and dual as 
\begin{equation}
\begin{aligned}\label{eq: pdipstepmaxlength}
     \mbox{PDIP:} \quad  \alpha_P^{\max} & = \min\left( 1, \max\{\alpha\mid x+\alpha \Delta x \geq 0\}\right), \\
     \mbox{PDIP:} \quad  \alpha_D^{\max} & = \min\left( 1,\max\{\alpha\mid s+\alpha \Delta s \geq 0\}\right),
\end{aligned}
\end{equation}
then choose the actual steplengths to be
\begin{equation} \label{eq: actual_pdipstepmaxlength}
    \mbox{PDIP:} \quad \alpha_P = \tau  \alpha_P^{\max},\quad \text{and}\quad 
    \alpha_D = \tau  \alpha_D^{\max},
\end{equation}
for some parameter $\tau\in [0,1)$. 
Typical values of $\tau$ are $.99$ and $.995$. 
(Values closer to $1$ are more aggressive.)
The updated $x$ iterate is then
$x^+ =  x + \alpha_P \Delta x$.
We make similar updates to $\lambda$ and $s$, using stepsize $\alpha_D$ for both.

For the stepsize in SSV-SQP, we again use two different stepsizes for primal variables $(x,v)$ and dual variables $(\lambda,s)$. We set the maximum step for primal and dual as 
\begin{equation}
\begin{aligned}\label{eq: sqp_stepmaxlength}
    \mbox{SSV-SQP:} \quad \alpha_P^{\max} & = \min\left( 1, \max\{\alpha\mid v+\alpha \Delta v \geq 0\}\right), \\
    \mbox{SSV-SQP:} \quad  \alpha_D^{\max} & = \min\left( 1,\max\{\alpha\mid s+\alpha \Delta s \geq 0\}\right).
\end{aligned}
\end{equation}
The actual step lengths are set to
\begin{equation} \label{eq: actual_sqpstepmaxlength}
    \mbox{SSV-SQP:} \quad \alpha_P = \tau \alpha_P^{\max},\quad \text{and}\quad 
    \alpha_D = \tau \alpha_D^{\max},
\end{equation}
where $\tau \in (0,1]$ is a fixed parameter, as for PDIP.
The primal variables $x$ and $v$ are updated using $\alpha_P$ while $\lambda$ and $s$ are updated with $\alpha_D$.
\rev{This choice of steplength maintains positiveness of elements of $v$, which ensures that our derivation in \eqref{eq:lp.sqp} and \eqref{eq: lp.sqp.lambda} results in a well-posed linear system at each iteration.}

We tried a version of SSV-SQP in which the value of $x$ was reset to $v \odot v$ as a final step within each iteration, but it was worse in practice than the more elementary strategy outlined here.

\paragraph{Extension to Nonlinear Programming}
Effective and practical primal-dual interior-point methods have been developed for nonlinearly constrained optimization; see in particular IPOPT \cite{WacB06}.
These are natural extensions of the PDIP and MPC techniques outlined in this section, but with many modifications to handle the additional issues caused by nonlinearity and nonconvexity.
We believe that the SSV-SQP approach described here could likewise be extended to nonlinear programming. 
The many techniques required to make SQP practical for {\em equality constrained} problems would be needed, with additional techniques motivated by our experience with linear programming above, for example, maintaining positivity of the squared variables.
One possible theoretical advantage is that if the SQP algorithm can be steered toward points that satisfy second-order conditions for the squared-variable (equality constrained) formulation, then our theory of \Cref{sec:ssv.theory} ensures convergence to second-order points of the original problem.
This contrasts with the first-order guarantees available for interior-point methods (and most other methods for nonlinear programming).
The practical implications of this theoretical advantage are however minimal.

\section{Computational Experiments} 
\label{sec:computation}

In this section, we test squared-variable formulations on four classes of problems: convex quadratic programming with bounds in \Cref{sec: cqpb}, linear programming in \Cref{sec:lp_computation}, \rev{nonnegative matrix factorization in \Cref{sec: NMF},} and constrained least squares in \Cref{sec: cls}. 
All experiments are performed using implementations in Matlab \cite{MATLAB}  on a 2020 Macbook Pro with an Apple M1 chip and 8GB of memory.

\subsection{Convex Quadratic Programming with Bounds}
\label{sec: cqpb}
We test the DSS approach on randomly generated convex quadratic programs with nonnegativity constraints. 
These problems have the form
\begin{equation} \label{eq:exp.1a}
    \min_{x \ge 0} \, f(x) := \frac12 x^TQx + b^Tx,
\end{equation}
where $Q$ is symmetric positive definite with eigenvalues log-uniformly distributed in the range $[\kappa^{-1},1]$ for a chosen value of $\kappa>1$, and eigenvectors oriented randomly. 
The vector $b$ is chosen such that the unconstrained minimizer of the objective is at a point $x_{\text{ref}}$, whose components are drawn i.i.d from the unit normal distribution $N(0,1)$.
We use the direct substitution $x = v \odot v$ to obtain the DSS formulation.
\begin{equation}
    \label{eq:exp.1b}
    \min_v \, F(v) := \frac12 (v \odot v)^T Q (v \odot v) + b^T (v \odot v).
\end{equation}
The \rev{common starting point} for the algorithms for \eqref{eq:exp.1a} is a vector with elements $\max(\phi,0)+1$, where $\phi$ follows the standard gaussian distribution.
Similarly, for \eqref{eq:exp.1b}, the initial $v$ is the vector whose elements are $\sqrt{\max(\phi,0)+1}$, where $\phi$ follows 
the standard gaussian distribution.\footnote{Note that we cannot
start from a point with $v_i=0$ for any $i$, since this component of the gradient will always be zero so a gradient method will never move away.}

We compare a gradient projection algorithm for \eqref{eq:exp.1a} with a gradient descent algorithm for \eqref{eq:exp.1b}.
The algorithms are as follows.
\begin{itemize}
    \item[] \textbf{PG.} Gradient projection for \eqref{eq:exp.1a}, where each step has the form $x^+ = \max(x-\alpha g,0)$, where $g = Qx+b$  is the gradient and $\alpha$ is a steplength chosen according to a backtracking procedure with backtracking factor $0.5$. The initial value of $\alpha$ at the next iteration is $1.5$ times the successful value of $\alpha$ at the current iteration. 
    \item[] \textbf{GD: Scaled.} Scaled gradient descent for \eqref{eq:exp.1b}. When $\| \nabla F(v) \| > .1$, the algorithm takes a normal backtracking gradient descent step, with parameters for manipulating $\alpha$ chosen as in PG. When $
    \| \nabla F(v) \| \le .1$, a unit step is tried in the direction $D^{-1} \nabla F(v)$, where $D = \diag (\nabla^2 F(v)) + \lambda I$, with $\lambda$ being the smallest positive value such that all elements of the diagonal matrix $D$ are at least $10^{-5}$. If this step yields a decrease in $F$, it is accepted; otherwise, a normal gradient descent step with backtracking is used.
\end{itemize}
Note that the cost of diagonal scaling is relatively small; it would be dominated in general by the cost of calculating the gradient. 
\rev{We also tested a version of PG for \eqref{eq:exp.1a} with similar diagonal Hessian scaling to GD: Scaled, but its results were similar to but slightly worse than PG, so we omit them. Likewise, we omit results for the non-scaled version of GD, whose performance was much worse than the scaled version of this algorithm.}

In all cases the algorithms are terminated when 
\[
\|x-\max(x-(Qx+b),0)\| \le \text{tol},
\]
for a small positive value of $\text{tol}$.
(For the formulation \eqref{eq:exp.1b}, we set $x = v \odot v$ for purposes of this test.)
\rev{The quantity $x-\max(x-(Qx+b),0)$ is the difference between consecutive iterates of PG when stepsize is $1$, and is also known as the proximal gradient mapping. 
If this condition is satisfied with $\text{tol}=0$ then $x$ is optimal.}

Results are shown in \Cref{fig: two_alg_QP}. We report on two values for condition number $\kappa$ (10 and 100) and two values of convergence tolerance $\text{tol}$ ($10^{-4}$ and $10^{-6}$), for dimension $n$ taking on five values between 100 and 2000. 
\rev{The plots show the number of iterations of each method, averaged over 25 random trials  in each case, together with the standard deviation (the shaded area) in iteration count over the trials.}
Importantly, the DSS formulation, despite its nonconvexity, always yielded a global solution  of the original problem.

\begin{figure}
    \centering
  \includegraphics[width=0.8\textwidth]{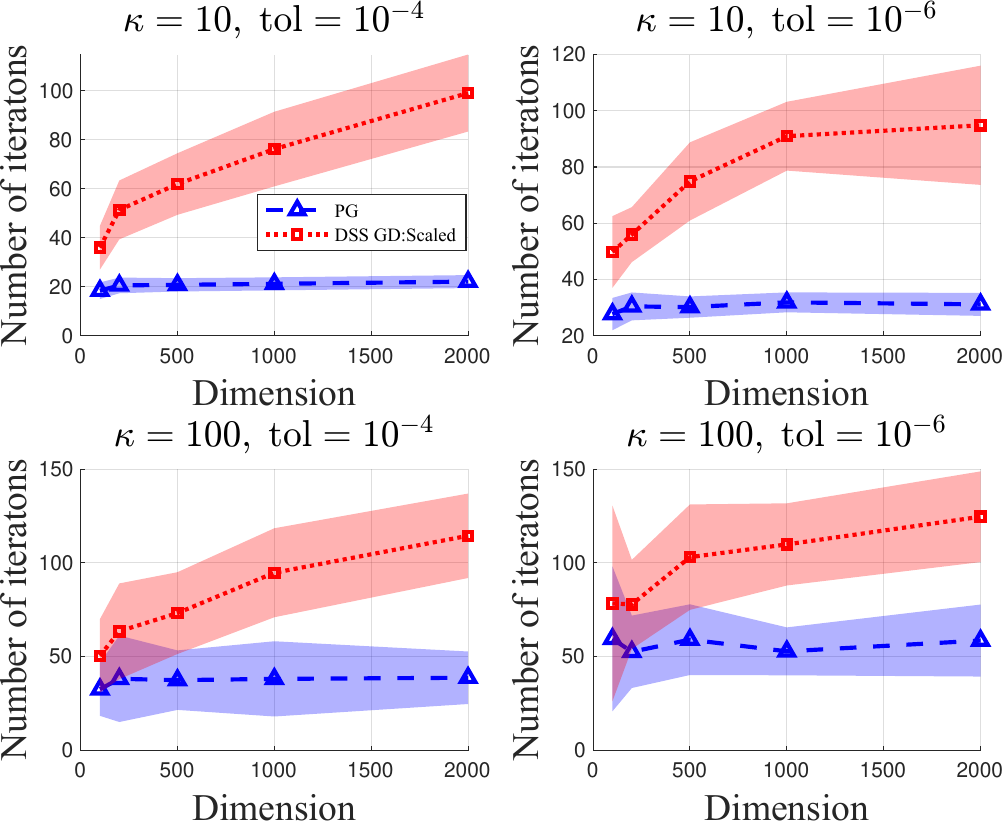}
    \caption{Iterations of gradient projection (PG) and scaled gradient descent (DSS GD: Scaled) applied to \eqref{eq:exp.1a}, compared with gradient descent and  diagonally-scaled gradient descent applied to \eqref{eq:exp.1b}, for several  values of dimension $n$, two values each of condition number $\kappa(Q)$, and convergence tolerance $\text{tol}$. Each line shows means (with plus and minus standard deviations in shade) over 25 random trials. \label{fig: two_alg_QP}
    }
\end{figure}
We note that for higher dimension and larger condition number, 
the PG approach is still faster than GD: Scaled, but by a (to us) surprisingly small factor, 
ranging from $1.2$ to about $3$.
This factor is lower (that is, DSS is more competitive) for smaller problems, larger condition number, and more accurate solutions.

\subsection{Linear Programming}
\label{sec:lp_computation}
We consider linear programs \eqref{eq:lp} with upper bounds on some variables:
\begin{equation}\label{eq:lpup} \tag{LP-u}
\min_x \, c^T x \;\; \mbox{s.t.} \;\; Ax=b, \;\; 0 \le x_{\cI} \le u, \;\; 0 \le x_{\cJ}.
\end{equation}
Here $\cI$ is a subset of $[n]$ and $\cJ$ is its complement. The vector $u\in \R^{|\cI|}$ is the vector of upper bounds. Upper bounds are common in real applications (see \Cref{sec:netlib}) 
and treating them explicitly can lead to more efficient implementations. 

By introducing variables $w$ for the upper bound constraints, we consider a form equivalent to \eqref{eq:lpup}:
\begin{equation}\label{eq:lpupw} \tag{LP-uw}
\min_{x,w} \, c^T x \;\; \mbox{s.t.} \;\; Ax=b, \;\; x_{\cI} +w= u, \;\; (x,w) \ge 0,
\end{equation}
which is in standard form for a linear program (linear equality constraints, with all variables nonnegative). 
In the SSV form of \eqref{eq:lpupw}, the nonnegativity requirements on $x$ and $w$ are replaced by equality constraints $x = v \odot v$ and $w = y \odot y$.
We apply the  MPC and SSV-SQP approaches introduced in \Cref{sec:lpssv}. 
Our solvers exploit the structure of the linear constraints in \eqref{eq:lpupw} that arise from the presence of upper bounds. 
Details are given in Appendix \ref{app: lpup}. 

\paragraph{Residuals and stopping criterion} To define the primal-dual residuals, we denote the dual variables for the constraints $x\geq 0$ and $w\geq 0$ as $s \in \real^{n}$ and $t\in \real^{|\cI|}$ respectively, and $\lambda$ is the dual variable for the linear constraint $Ax = b$. 
Define the following residuals:
\begin{subequations} \label{eq:lpupw.residual}
\begin{align}
\label{eq:lpupw.residual.1a}
     r_{c_\cI} & := -\left(c_\cI -s_\cI +t -A_\cI ^\top \lambda\right) , 
    & r_{c_\cJ} & := -\left(c_\cJ -s_\cJ -A_\cJ ^\top \lambda\right) \\
    \label{eq:lpupw.residual.1c}
     r_x & := -\left(b-Ax\right), 
     &r_u  &:= -\left(u -x_\cI -w\right), \\
    \label{eq:lpupw.residual.1e}
    r_{xs} &:=  x\odot s,
   &r_{rw} &:=  t\odot w.
 \end{align}
\end{subequations}
The residual used to define the stopping criteria is defined as
\begin{equation}\label{eq: lp.res}
    \text{res} :\,= \frac{\norm{(r_{c_\cI}, r_{c_\cJ},r_x, r_u, r_{xs}, r_{rw},x_-,w_-)}}{1+\max\{\norm{b},\norm{c}\}}.
\end{equation}
Here $x_-,w_-$ are the negative parts of $x$ and $w$, respectively.
Note that we did not consider the nonnegativity of the dual variables $s$ and $t$ because both MPC and SSV-SQP maintain nonnegativity of these variables. 
We terminate the algorithm when $\text{res}\leq \epsilon$ for some threshold $\epsilon>0$, or when it reaches $500$ iterations, or $750$ seconds of wall-clock time, whichever comes first.

\paragraph{Parameter choice and initialization} 
For the step scaling parameter $\tau$, we choose $\tau =\{0.15,0.5,0.75,0.9\}$ for SSV-SQP and $\tau = \{0.9, 0.95,0.995\}$ for MPC. For MPC, to initialize $(x,w,\lambda,s,t)$, we define
$M := \max\{\infnorm{A},\infnorm{b},\infnorm{c}\}$ and set
$x_0 = s_0 = 100 M \textbf{1}_n $ (where $\textbf{1}_n$ is the vector of length $n$ whose elements are all $1$), $w_0= t_0 = 100 M \textbf{1}_{|\cI|}$, and $\lambda = 0$. For SSV-SQP, we initialize $(x,w,\lambda,s,t)$ in the same way as MPC.
The initial value of the squared variable is taken to be the square root of the corresponding original variable, calculated entrywise.

In the following sections, we  report only the number of iterations for MPC and SSV-SQP, not the runtime. 
The actual runtime for the algorithms to reach a certain accuracy is usually more favorable to SSV-SQP, since it only needs to solve one linear system per iteration, while our naive implementation of MPC factors the coefficient matrix twice at each iteration.
Our codes for MPC and SSV-SQP are identical to the extent possible; \rev{both are based on  PCx \cite{czyzyk1999pcx}.}



\subsubsection{Random Linear Programs}
\label{sec: RLP}
We compare the performance of MPC and SSV-SQP on random instances of \eqref{eq:lpup}. 
We choose $n \in \{500,1000,2500\}$, while  $m = 0.1n$, $0.25n$, and $0.5n$. 
The number of upper bounds $| \cI|$ is set to be $\lfloor 0.05m \rfloor$.
We generate the upper-bounded component indices $\cI$ by randomly drawing from  the set $\{1,2,\dotsc,n\}$ without replacement, and choose the upper bounds on these components randomly from a uniform distribution on $[1,21]$.
For each combination of $(m,n)$, we randomly generate ten trials of $(A,b,c,\cI,u)$.  
The matrix $A$, vector $c$, and primal feasible point $\tilde{x}$  have entries drawn randomly and i.i.d. from a uniform distribution over $[0,1]$. We set  $b := A \tilde{x}$.

For each choice of parameters and dimensions, we report the average number of iterations over 10 trials to reach a residual smaller than $10^{-8}$ (together with the standard deviation over the 10 trials) for SSV-SQP with $\tau = .5,.75,.9$ and MPC with $\tau = .995$. \footnote{Since SSV-SQP with $\tau = .15$ is slow for larger $n$ and $m$, we do not report the results here. Among the values of $\tau$ tried for MPC,  $\tau =.995$ was generally superior.}

Results for random LPs are shown in Table \ref{tb: comparisonMPCvsSSV}.  The symbol $*$ indicates that the method fails to reach the desired accuracy in at least one of the ten trials.

\begin{table}[h]
\centering
\begin{tabular}{|c|c|c|c|c|}
\hline
$(n,m)$ & MPC & SSV ($\tau=.5$)  &  SSV ($\tau=.75$)  & SSV ($\tau=.9$)  \\  
\hline 
\hline
(500,50) & 15.5 (0.71) & 61.0 (2.5) & 38.5 (1.8) & $*$ \\ 
\hline
(500,125) & 15.8 (0.42) & 60.9 (2.4) & 39.2 (1.8) & $*$ \\ 
\hline
(500,250) & 15.9 (0.32) & 61.3 (3.4) & 39.4 (2.2) & 32.0 (2.2) \\ 
\hline
(1000,100) & 17.1 (0.74) & 62.6 (2.2) & 39.8 (1.5) & 32.5 (1.4) \\ 
\hline
(1000,250) & 17.4 (0.70) & 64.8 (3.1) & 42.3 (2.9) & 35.1 (2.8) \\ 
\hline
(1000,500) & 17.7 (0.48) & 64.3 (1.6) & 42.4 (1.0) & 36.1 (1.4) \\ 
\hline
(2500,250) & 18.2 (0.63) & 67.4 (3.4) & 43.7 (2.3) & 37.0 (2.6) \\ 
\hline
(2500,625) & 18.7 (0.68) & 70.1 (2.8) & 46.3 (2.2) & 38.1 (1.4) \\ 
\hline
(2500,1250) & 18.7 (0.48) & 68.3 (2.2) & 45.4 (2.1) & 39.4 (2.8) \\ 
\hline
\end{tabular}
\caption{Performance of  MPC and SSV-SQP for three different values of the parameter $\tau$. Each entry shows average number of iterations together with the standard deviation over ten trials. Missing entries correspond to cases in which convergence was not attained in one or more of the trials.} \label{tb: comparisonMPCvsSSV}
\end{table}

We note from this table that SSV is consistently slower than MPC, by a factor of 2-4 in the number of iterations.
However, this factor was surprisingly  modest to us, given the relatively high sophistication and widespread use of the MPC approach (despite our rather elementary implementation), and the lack of tuning of the SSV method.\footnote{\rev{Our surprise also comes from a naive SSV approach to \eqref{eq:lpupw}: using  \texttt{fmincon} in Matlab \cite{ma2021beyond} to solve the SSV form of \eqref{eq:lpupw}. The performance of SSV-SQP  is much better than that approach, e.g., SSV-SQP with $\tau = 0.75$ is almost always faster than  \texttt{fmincon} by a factor of $10$ or more.}
}For SSV, the less aggressive value $\tau=.5$ was more reliable in finding a solution, whereas the more aggressive value $\tau=.9$ required approximately half as many iterations, but was less reliable, with the iterates diverging for some trials.
(Similar phenomena were noted for non-random problems, as we report below.)
The standard deviation of iteration count across trials was slightly higher for SSV than for MPC.

It is possible that the SSV approach could be made significantly more efficient with more sophisticated techniques in SSV-SQP, including an adaptive, varying scheme for choosing the step scaling parameter $\tau$, and other modifications to the basic SQP approach.

   \subsubsection{Netlib Test Set}
   \label{sec:netlib}
We test the performance of SSV-SQP against MPC on  the celebrated Netlib test set \cite{gay1985electronic}.
which played an important role in the development of interior-point codes for linear programs in the period 1988-1998.

Since the problem data of Netlib usually have some redundancies and degeneracies, which hurts code performance and makes comparisons unreliable, we use the presolver from PCx \cite{czyzyk1999pcx} on each problem.
Presolvers typically reduce the problem sizes, eliminate redundancies of many types, and improve the conditioning. This presolver outputs problems of the upper-bound form \eqref{eq:lpup}.

We report the number of instances solved to several accuracies $\epsilon$ for SSV-SQP and MPC in Table~\ref{table: SSV-MPC-netlib-suc-count}.\footnote{We report the results for MPC with $\tau =.90$ as it solves most of the problems among the choices of $\tau$ and only requires $\le 5$ iterations more than other choices of $\tau$. }. 
We also report the average number of iterations for the solved instances in Table~\ref{table: SSV-MPC-netlib-suc-mean}. 
A detailed report on the performance of SSV-SQP and MPC  for each instance in Netlib with two of these accuracy levels --- $10^{-2}$ and $10^{-5}$ --- can be found in 
\Cref{sec: netlib_MPCvsSSV}.

\begin{table}[h]
\centering 
\begin{tabular}{|c||cccccccc|}
\hline 
 $\epsilon$ & $10^{-1}$ & $10^{-2}$ & $10^{-3}$ & $10^{-4}$ &  $10^{-5}$ & $10^{-6}$ & $10^{-7}$ & $10^{-8}$ \\ 
\hline \hline 
$\tau = 0.15$ & 88 & 88 & 87 & 85 & 83 & 73 & 50 & 22\\ 
\hline 
$\tau = 0.5$ & 88 & 86 & 85 & 83 & 71 & 49 & 26 & 12\\ 
\hline 
$\tau = 0.75$ & 88 & 86 & 81 & 71 & 54 & 28 & 14 & 10\\ 
\hline 
$\tau = 0.9$ & 86 & 82 & 72 & 51 & 31 & 18 & 8 & 2\\ 
\hline \hline
MPC & 88 & 86 & 82 & 81 & 78 & 75 & 72 & 68 \\
\hline 
\end{tabular}
\caption{The number of cases solved to accuracy $\epsilon$ among $94$ instances for Netlib problems: SSV-SQP with $\tau = 0.15$, $0.5$, $0.75$, $0.9$;  and MPC with $\tau = 0.9$ (last row).}
\label{table: SSV-MPC-netlib-suc-count}
\end{table}

\begin{table}[h]
\centering 
\begin{tabular}{|c||cccccccc|}
\hline 
$\epsilon$ & $10^{-1}$ & $10^{-2}$ & $10^{-3}$ & $10^{-4}$ &  $10^{-5}$ & $10^{-6}$ & $10^{-7}$ & $10^{-8}$ \\ 
\hline 
\hline
$\tau = 0.15$ & 177.2 & 212.2 & 242.0 & 274.2 & 288.3 & 315.8 & 328.4 & 332.3 \\
\hline
$\tau = 0.5$ & 58.3 & 67.3 & 79.5 & 88.5 & 96.7 & 94.4 & 86.8 & 109.6 \\
\hline
$\tau = 0.75$ & 41.5 & 52.8 & 59.6 & 72.4 & 73.4 & 53.1 & 55.4 & 58.1 \\
\hline
$\tau = 0.9$ & 42.8 & 55.1 & 60.2 & 56.4 & 47.3 & 45.8 & 38.6 & 27.0 \\
\hline \hline
MPC  & 18.8 & 21.3 & 23.2 & 25.3 & 27.3 & 28.5 & 30.1 & 30.8 \\
\hline
\end{tabular}
\caption{The average number of iterations among the solved instance with accuracy $\epsilon$ for Netlib problems: SSV-SQP with $\tau = 0.15$, $0.5$, $0.75$, $0.9$; and MPC with $\tau = 0.9$.}
\label{table: SSV-MPC-netlib-suc-mean}
\end{table}

Table~\ref{table: SSV-MPC-netlib-suc-count} shows that as the step-scaling parameter $\tau$ decreases, SSV-SQP is able to solve more problems but requires more iterations. 
SSV-SQP with $\tau =.15,.5,.75$ is even able to solve more problems than MPC for less stringent  tolerances $(10^{-1}$-$10^{-2})$. 
SSV-SQP with $\tau = .5$ is also able to solve at least as many instances as MPC for accuracies $10^{-1}$-$10^{-4}$.
Table~\ref{table: SSV-MPC-netlib-suc-mean} shows that for $\tau=0.5$, the number of iterations of SSV-SQP is within a factor of $4$ of that for MPC. 
For $\tau=0.9$, among solved instances, SSV-SQP is even more competitive with MPC, but for smaller values of $\epsilon$, many problem instances are not solved by SSV-SQP.
The number of iterations required by SSV-SQP for the most conservative setting $\tau=.15$ is much worse, averaging a factor of about 10 greater than MPC.



We note the low success rates for SSV-SQP with higher accuracies (smaller $\epsilon$). 
\rev{Overall, although the results for SSV-SQP are considerably worse overall  than for MPC, we see reasons to investigate the approach further. First, the difference in performance and robustness of the two approaches is least for low-accuracy solutions, and low accuracy suffices for some applications. Second,
we believe that tuning we could improve the performance of SSV-SQP, for which (by contrast with MPC) there is no ``folklore" concerning heuristics that improve efficiency.}
\footnote{In the experiments, we observed that SSV-SQP fails to converge when the complementarity measure $\|r_{xs}\|_1+\|r_{rw}\|_1$ is small while the infeasibility measure increases. 
Convergence tends to occur when these two measures decrease at similar rates.
Modifying SSV-SQP steplength strategy to ensure this property could be the key to improving the algorithm.}

\rev{
\subsection{Nonnegative matrix factorization (NMF)}
\label{sec: NMF}
We consider nonnegative matrix factorization (NMF), a nonconvex problem type that arises in a variety of applications (for example, recommendation systems \cite{luo2014efficient}). 
We consider the formulation
\begin{equation}\label{eq: nmfOG}
    \min_{X\in \mathbb{R}^{n\times r}, X\geq 0} f(X):=\fronorm{XX^\top - M}^2,
\end{equation}
where $X \ge 0$ indicates element-wise nonnegativity.
We test the DSS approach on randomly generated problems in which $M = UU^\top$ for some $U\in \mathbb{R}^{n\times r}$ and $U\ge 0$, where
each entry of $U$ is drawn uniformly from $[0,1]$. 
We use the  direct substitution $X = V\odot V$ for some $V\in \mathbb{R}^{n \times r}$ to obtain the DSS formulation:
\begin{equation}\label{eq: nmfDSS}
    \min_{V \in \mathbb{R}^{n\times r}} F(V):=\fronorm{(V\odot V)(V\odot V)^\top - M}^2
\end{equation}
Note that the optimal objective value in this problem is $0$, since there is a solution $V$ whose elements are the square roots of the corresponding element of $U$. 
We consider five methods for \eqref{eq: nmfDSS}: 
\textbf{PG} for \eqref{eq: nmfOG} and \textbf{GD} for \eqref{eq: nmfDSS}, both with backtracking line search, 
as described in Section \ref{sec: cqpb}.  For better numerical performance, the backtracking factor $\alpha$ is $0.35$ and the initial $\alpha$ at the next iteration is 2 times the successful value of $\alpha$ at the current iteration. In addition, we consider the following three methods:
\begin{itemize}
\item \textbf{PG-Polyak.} Gradient projection with the stepsize, $\frac{f(X)} {2\fronorm{\nabla f(X)}^2}$ for \eqref{eq: nmfOG}.
\item \textbf{GD-Polyak.} Gradient descent with the stepsize, $\frac{F(V)}{2\fronorm{\nabla F(V)}^2}$ for \eqref{eq: nmfDSS}.
\item \textbf{L-BFGS.} L-BFGS implemented by \texttt{fminunc} in Matlab for \eqref{eq: nmfDSS}.
\end{itemize}
The step size for \textbf{PG-Polyak} and \textbf{GD-Polyak} is half the Polyak stepsize. 
We use a more conservative stepsize because the problem is nonconvex. 
We declare convergence when $\text{acc} \le \epsilon$, where
\begin{equation}
    \text{acc} :=  \frac{\fronorm{XX^\top -M}^2}{\fronorm{M}^2},
\end{equation}
%
%
%
%
and  $\epsilon = 10^{-2}$ or $10^{-4}$. 
(For the formulation \eqref{eq: nmfDSS}, we set $X = V\odot V$.) We choose dimension 
$n \in \{1000,2000,3000,4000,5000\}$ and the rank $r \in \{10,\,20\}$. 
For each combination of $r$ and $n$, we generate $20$ random versions of $U$ 
and record the number of iterations required by each of the two approaches to reach the required accuracy.  

The results are shown in \Cref{fig: nmf}. 
PG based on the original formulation \eqref{eq: nmfOG} is generally more effective for low accuracy $(\epsilon=10^{-2})$  than the methods based on \eqref{eq: nmfDSS}. 
For medium accuracy $(\epsilon=10^{-2})$, the methods based on \eqref{eq: nmfDSS} are generally better.
For example, GD-Polyak improves on PG and PG-Polyak by factors of $1.5-5$. 
The number of iterations of GD-Polyak, PG, and PG-Polyak ranges from $250-2500$, so the saving in absolute terms is significant. 
The method L-BFGS applied to \eqref{eq: nmfDSS}  is highly effective in this regime, improving on all other methods by factors of $3-15$. 
In additional experiments (not shown here) with $\epsilon=10^{-5}$, L-BFGS is even better. 

These results demonstrate the potential of the squared-variable approaches for constrained matrix optimization problems. They may be an alternative or complement to existing standard approaches.

\begin{figure}
    \centering
\includegraphics[width= 0.9\textwidth]{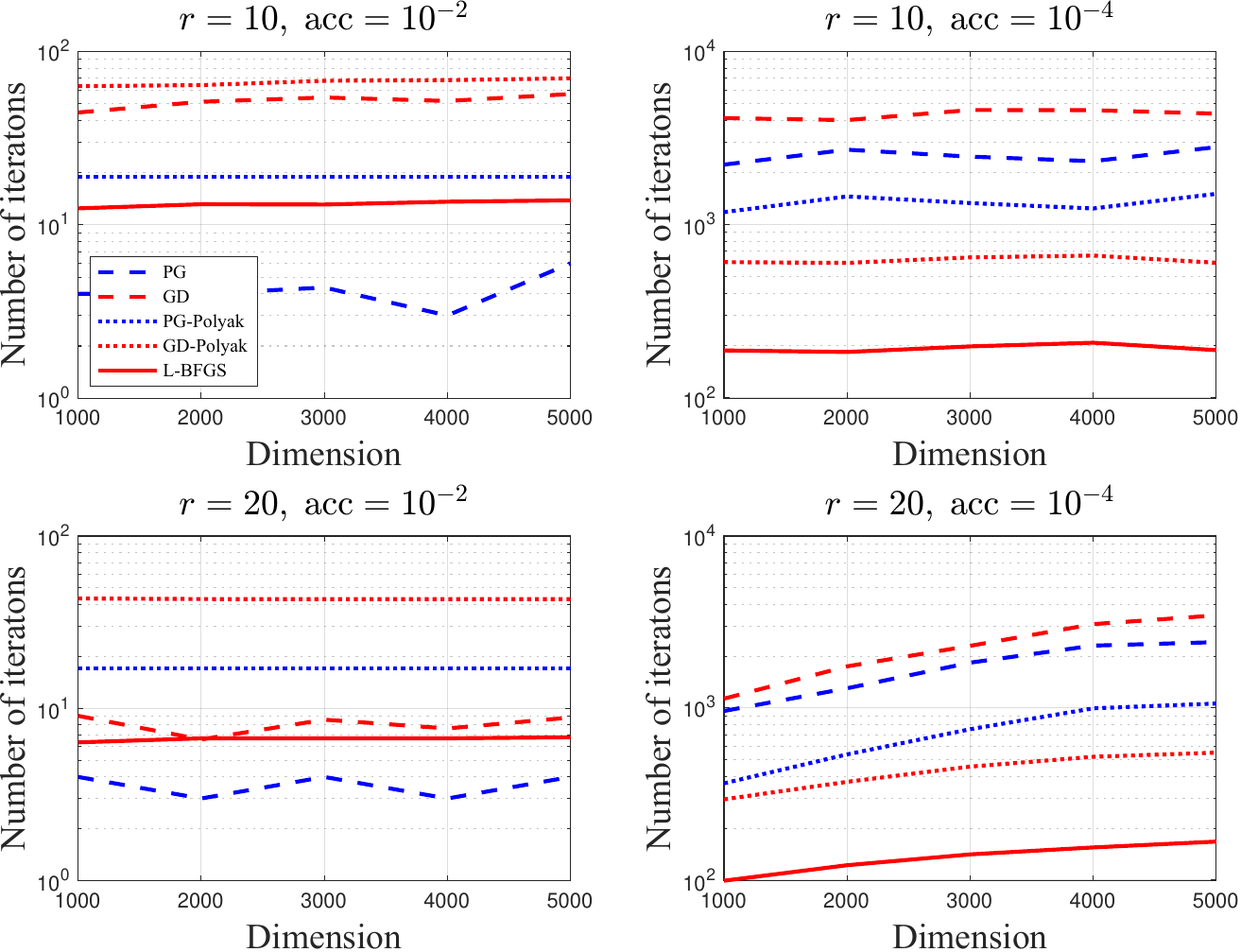}
    \caption{Number of iterations to reach the required accuracy for nonnegative matrix factorization. The dimension varies from $1000$ to $5000$, and $r = 10$ and $20$. The lines are averaged over $20$ random trials. Red lines are the algorithms based on \eqref{eq: nmfDSS} while blue lines are based on \eqref{eq: nmfOG}. Note the $y$-axis is on a log scale. 
    \label{fig: nmf}}
\end{figure}
}

\subsection{Constrained Least Squares}
\label{sec: cls}
In this subsection, we consider the constrained least square problem with a sparse solution:
\begin{equation}
    \label{eq:exp.2a}
    \min \, \|Ax-b \|_2^2 /2\quad \mbox{s.t.} \; \|x\|_\tau \le R_{\tau}^{\frac{1}{\tau}},
\end{equation}
where $A \in \R^{m \times n}$ with $m<n$ and $R_{\tau}>0$ is given. 
Here the pseudo-norm $\|\cdot\|_\tau$ is defined as  $\|x\|_\tau = (\sum_{i=1}^\dm |x_i|^\tau)^{\frac{1}{\tau}}$. 
We choose  $\tau \in (0,1]$ to obtain a sparse solution. 
For $\tau = 1$, the problem \eqref{eq:exp.2a} is the LASSO problem \cite{tibshirani1996regression}.

Substituting $x = v \odot v - w \odot w $ and using the fact that in order to represent $x_i$, only one of $v_i$ and $w_i$ needs to be nonzero, we see that \eqref{eq:exp.2a} is equivalent to
\begin{equation}
    \label{eq:exp.2b}
    \min \, \|A(v \odot v - w \odot w)-b \|_2^2 /2\quad \mbox{s.t.} \; \sum_{i=1}^n (v_i^2)^\tau + (w_i^2)^\tau  \le R_\tau.
\end{equation}
For $\tau =1$, the feasible region of \eqref{eq:exp.2b} is an $\ell_2$-norm ball.
Projection onto this set  can be performed in $\mathcal{O}(n)$ time. 
Using the analysis in \cite{li2021simplex} or extending our analysis here, one can show  that 2NP for \eqref{eq:exp.2b} are 1P for \eqref{eq:exp.2a} and therefore optimal for \eqref{eq:exp.2a}.

Projecting to the feasible set for  \eqref{eq:exp.2b} is difficult for general values of  $\tau$ in $[0,1)$. To deal with this issue, we consider $\tau = 1/L$ for even integer values of  $L$. 
Using the substitution $x = \underbrace{v\odot \dots \odot v}_{L\;\text{times}} -\underbrace{w\odot \dots \odot w}_{L\;\text{times}} 
= v^{\odot L} -w ^{\odot L}$, and the fact that in order to represent $x_i$, only one of $v_i$ and $w_i$ needs to be nonzero, we obtain an equivalent problem of \eqref{eq:exp.2a} (in terms of global solution):
\begin{equation}
    \label{eq:exp.2c}
    \min \, g_L(v,w):=  \|A(v^{\odot L} - w ^{\odot L})-b \|_2^2 /2\quad \mbox{s.t.} \; \sum_{i=1}^n |v_i| + |w_i|  \le R_{1/L}.
\end{equation}
Note that the objective, in this case, is still smooth and the constraint set is an $\ell_1$ norm ball, for which projection is efficient.

We now explain how $A$ and $b$ are generated and what algorithm we use to solve \eqref{eq:exp.2b} and \eqref{eq:exp.2c}. In our experiments, $A$ has i.i.d. standard Gaussian entries. We generate a sparse vector $\beta_0\in \real^n$ with given sparsity $s$ whose nonzero entries are drawn from the uniform distribution over $[-1,1]$. Setting a  noise level  $\sigma =1$, we define $b = A\beta_0 + e_{\sigma}$, where $e_{\sigma} \sim N(0,\sigma^2 I)$.
We choose $\tau = \left\{1, \frac{1}{2},\frac{1}{4},\frac{1}{6},\frac{1}{8} \right\}$ (corresponding to $L=1,2,4,6,8$), in order.
For each value of $\tau$, the parameter $R_\tau$ is chosen to be the value appropriate to the reference solution $\beta_0$, that is, $R_\tau = \sum_{i=1}^n |\beta_{0i}|^\tau$. 
We use projected gradient (PG) for solving \eqref{eq:exp.2b} and \eqref{eq:exp.2c} with Armijo line search  and terminate  the algorithm when
\begin{equation} \label{eq:fskhj7}
    \twonorm{(v,w) - \mathcal{P}\left[(v,w)-\nabla g_L(v,w)\right]}\leq 10^{-6}.
\end{equation} 
Here $\mathcal{P}$ is the projection operator of the $\ell_2$ norm ball with radius $R_{1}$ in $\real^{2n}$ for $\tau =1$, and $\mathcal{P}$ is the projection operator of the $\ell_1$ norm ball with radius $R_{1/L}$ in $\real^{2n}$  for other choices of $\tau$.
For each value of $\tau$, we let PG run for at most $T=200$ iterations, or when the test \eqref{eq:fskhj7} is satisfied, whichever comes first.
We start with a random initialization for $\tau=1$.
Then, \rev{for} each subsequent value of $\tau$,  we initialize at the final iterate of PG for the previous value of $\tau$.

In Figure~\ref{fig: RandomUniformTerm}, we present the numerical results of the relative objective value $g_L(v,w)/\twonorm{b}^2$ and  relative recovery error $\twonorm{v^{\odot L}-w^{\odot L}-\beta_0}/\twonorm{\beta_0}$ against the number of iterations, averaged over $500$ trials.
A marker ``$+$" appears on each trace in these figures corresponding to the average number of iterates required for termination, for each value of $\tau$.
After each of the $500$ trials terminates, the values of $g_L(v,w)/\twonorm{b}^2$ and  $\twonorm{v^{\odot L}-w^{\odot L}-\beta_0}/\twonorm{\beta_0}$ are taken to be their respective values at the iterate of termination.
In these tests, we set $n=1000$, $m=207$, $s=5$, and $\sigma=1$. .

Note that the use of pseudo-norms with $\tau<1$ significantly improves over the $\tau=1$ case in terms of closeness of the recovered  solution to the reference solution $\beta_0$. 
The final recovery error is much smaller for  $\tau = 1/2$ than for $\tau=1$, and slightly smaller again for $\tau=1/4$, but does not decrease significantly for smaller values of $\tau$. 
The warm-starting strategy leads to relatively few PG iterations being required for values of $\tau$ less than one, once the solution has been calculated for $\tau=1$.

\begin{figure}
    \centering
    	\begin{subfigure}[(a)]{.49\textwidth}
		\centering
	\includegraphics[width=1\linewidth, height= 0.25\textheight]{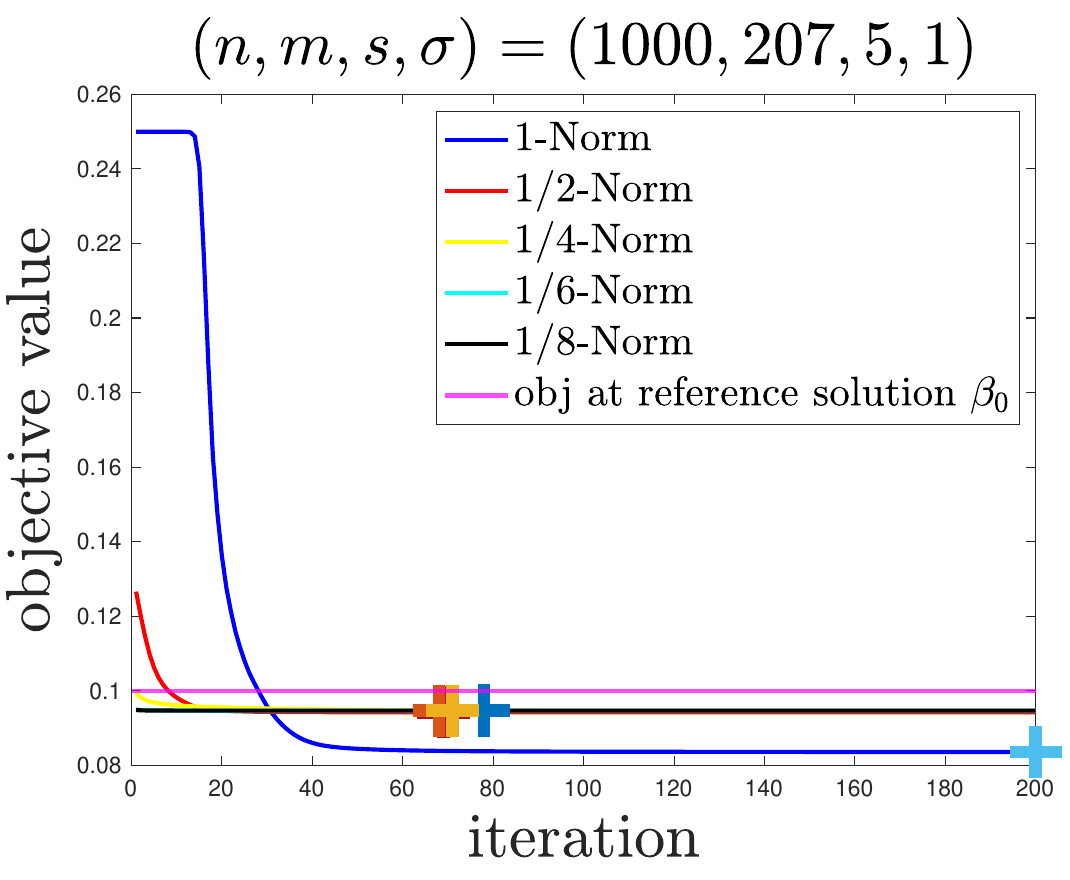}
		\caption{Objective value against the number of iterations}
		\label{fig: objective value RandomUniformTerm}
	\end{subfigure}
 \begin{subfigure}[(b)]{.49\textwidth}
		\centering
\includegraphics[width=1\linewidth, height= 0.25\textheight]{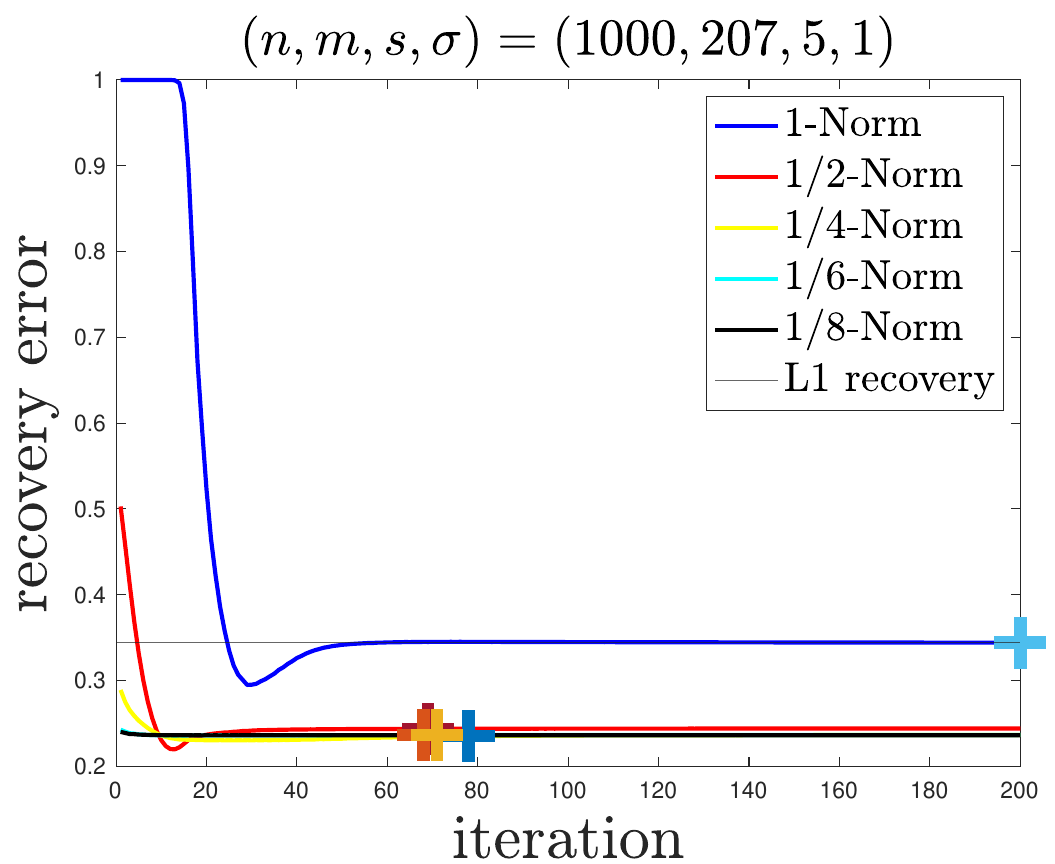}
		\caption{Recovery error against the number of iterations}
		\label{fig: recovery value RandomUniformTerm}
	\end{subfigure}
    \caption{Performance of pseudo-norm recovery using a power-variable formulation with warm starts. Average over $500$ trials. Each trial's reference solution $\beta_0$ has five nonzero components drawn from the uniform distribution over $[-1,1]$.  The last iteration for each value of $\tau$ is averaged over all trials, with the average marked as $+$ in the plot. 
    \label{fig: RandomUniformTerm}}
\end{figure}

\section{Conclusions}
We have considered both theoretical and practical aspects of squared-variables formulations for inequality-constrained optimization problems. 
Building on previous work, we find that some of the theoretical drawbacks  are less problematic than conventional wisdom would suggest.
Practically, although squared-variable approaches do not improve on the state of the art in quadratic and linear programming, they are more competitive than might be expected, \rev{and may become more competitive with further development, at least in some contexts.}  \rev{The results for nonnegative matrix factorization already show improvement over the state of the art.}
In the case of pseudo-norm constrained least squares (a problem of interest in sparse recovery) an extension of the squared-variable approach has good computational performance \rev{and better results than the standard approach based on the $\ell_1$ norm.}

\section*{Acknowledgments}
We thank Ivan Jaen Marquez for his help in resuscitating the PCx linear programming code for use in our experiments.

\bibliographystyle{siamplain}
\bibliography{reference}


\appendix

\rev{
\section{Equivalence of 1Ps of \eqref{eq: L1} and \eqref{eq: L1bc}}
\label{sec: 1Ps}
\begin{lemma}\label{lem: 1PL1L1bc}
Consider \eqref{eq: L1} and \eqref{eq: L1bc} with $\lambda>0$. 
Then if $x$ is a 1P of \eqref{eq: L1}, then \eqref{eq: L1bc} admits a 1P $\bar{x} = (x^+,x^-)$, where $x^+$, $x^-$ are the positive and negative parts of $x$, respectively.
Conversely, if $\bar{x} =(x^+,x^{-}) $ is a 1P of \eqref{eq: L1bc}, then $x := x^+-x^-$ is a 1P of \eqref{eq: L1}.
\end{lemma}
\begin{proof}
    Conditions for $x$ to be  a 1P of \eqref{eq: L1} are 
    \begin{equation}\label{eq: 1PL1}
     [\nabla h(x)]_i = \begin{cases}
    -\lambda, & x_i >0 ,\\ 
    \alpha_i \in [-\lambda, \lambda],& x_i  = 0,\\
    \lambda,  & x_i <0.
    \end{cases}
    \end{equation}
    for some $\alpha_i\in [-\lambda, \lambda]$.
    A point $\bar{x}$ is a 1P of \eqref{eq: L1bc}  if it satisfies the 1N conditions in Definition \ref{def:bc}: (i) nonnegativity: $\bar{h}(\bar{x})\geq 0$ and $\bar{x}\geq 0$, and (ii) complementarity: $\nabla \bar{h}(\bar{x})_i  \bar{x}_i = 0$ for $i=1,\dots,2n$. Here, the gradient $\nabla \bar{h}(\bar{x})$ admits the formula: 
    \begin{equation} \label{eq: L1bc1P}
    [\nabla \bar{h}(\bar{x})]_j = 
    \begin{cases}
        [\nabla h(x)]_j + \lambda , & j \leq n\\
        -[\nabla h(x)]_{j-n} +\lambda, & j>n
    \end{cases}.
    \end{equation}

    If $x$ is a 1P of \eqref{eq: L1}, then $\bar{x}\ge 0$ holds by definition, and $\nabla \bar{h} (\bar{x})\geq 0$ due to \eqref{eq: 1PL1} and  \eqref{eq: L1bc1P}. 
For the complementarity condition, if $x_i=0$, then $x_{i}^+=x_{i}^-=0$, and clearly $\nabla h(\bar{x})_j \bar{x}_j =0$ for $j = i,n+i$. If $x_i>0$, then 
    $x_{i}^+ >0$, $x_{i}^- =0$, and 
    $[\nabla h(\bar{x})]_i = 0$ due to \eqref{eq: 1PL1} and 
     \eqref{eq: L1bc1P}. Hence $\nabla h(\bar{x})_j \bar{x}_j =0$ for $j=i,n-i$. 
     (The case $x_i<0$ is similar.)

     Conversely, let $\bar{x}=(x^+,x^-)$ be a 1P of \eqref{eq: L1bc}. 
     We will show $x= x^+-x^-$ is a 1P of \eqref{eq: L1}. 
     If we have both $x_{i}^+ >0$ and $x_{i}^->0$, then $[\nabla\bar{h}(\bar{x})]_i =0$ and $[\nabla \bar{h}(\bar{x})]_{n+i}=0$ from the complementarity condition. Using \eqref{eq: L1bc1P}, we see this is impossible because $\lambda \not =0$. 
     Thus, if $x_{i}^+>0$, then $x_i = x_{i}^+$, and  $[\nabla h(x)]_j = -\lambda$ from the complementarity condition.
     Thus we see \eqref{eq: 1PL1} holds for $x_i$ in thsi case. 
     A similar argument shows \eqref{eq: 1PL1} for $x_{i}^->0$. 
     Finally, if both $x_{i}^+=x_{i}^-=0$, then \eqref{eq: L1bc1P} with $\nabla \bar{h}(\bar{x})\geq 0$ implies that $[\nabla h(x)]_i \in [-\lambda, \lambda]$, that is, \eqref{eq: 1PL1} holds. 
     Our proof is complete. 
\end{proof}
}


\section{PDIP, MPC, and SSV for LP with upper bounds}\label{app: lpup}
We derive the formulas of one iteration for PDIP, MPC, and SSV-SQP for \eqref{eq:lpupw}.

\subsection{PDIP}\label{sec: app.pdip}
We introduce dual variables $\lambda \in \RR^m$ for $Ax=b$, $s \in \RR^n$ for $x\geq 0$, $t \in \RR^{|\cI|}$ for $w\geq 0$. The KKT condition for \eqref{eq:lpupw} is 
\begin{subequations} \label{eq:lpupw.kktip}
\begin{align}
\label{eq:lpupw.kktip.1a}
    A^T_{\cI} \lambda + s_\cI - t - c_\cI &= 0, \\
    \label{eq:lpupw.kktip.1b}
    A^T_{\cJ} \lambda + s_\cJ  - c_\cJ &= 0, \\
    \label{eq:lpupw.kktip.2a}
    Ax - b &=0, \\
    \label{eq:lpupw.kktip.2b}
    x_{\cI} +w &=u, \\
    \label{eq:lpupw.kktip.3}
    s \odot x &= 0,\quad w\odot t =0 \\
    \label{eq:lpupw.kktip.4}
    (x,s,t,w)& \ge 0.
\end{align}
\end{subequations}
Here $A_{\cI}$ and $A_{\cJ}$ are columns of $A$ in $I$ and $J$ respectively.

Each iteration of PDIP for \eqref{eq:lpupw.kktip} aims to solve:
\begin{subequations} \label{eq:lpupw.pdip}
\begin{align}
\label{eq:lpupw.pdip.1a}
    A^T_{\cI} \Delta \lambda + 
    \Delta s_\cI - \Delta r &= -r_{c_\cI} :\,= c_\cI -s_\cI +t -A_\cI \lambda , \\
    \label{eq:lpupw.pdip.1b}
    A^T_{\cJ} \Delta \lambda + 
    \Delta s_\cJ &= -r_{c_\cJ} :\,= c_\cJ -s_\cJ -A_\cJ \lambda \\
    \label{eq:lpupw.pdip.2a}
    A\Delta x  &=  -r_x:\, = b-Ax, \\
    \label{eq:lpupw.pdip.2b}
    \Delta x_{\cI} +\Delta w &=-r_u :\;= u -x_\cI -w, \\
    \label{eq:lpupw.pdip.3a}
    X\Delta s + S\Delta x  & = -r_{xs}:\;= - x\odot s + \sigma (s^Tx+w^Tt)/(n+|\cI|),\\ 
     \label{eq:lpupw.pdip.3b}
    W\Delta r + T\Delta W  & = -r_{rw}:\;= - t\odot w+ \sigma (s^Tx+w^Tt)/(n+|\cI|),
 \end{align}
\end{subequations}
where $X = \diag(x)$, $S = \diag(s)$, $W= \diag(w)$ and $T= \diag(t)$, and $\sigma \in [0,1]$ is some centering parameter. All components of $x$, $s$, $t$, and $w$ are maintained strictly positive throughout.

Let $r_{xwr} \in \RR^n$ such that 
\begin{equation} 
r_{xwr,\cI} = X_{\cI}^{-1} r_{xs,\cI} -W^{-1} r_{rw} +W^{-1}T r_u -r_{c_\cI},
\end{equation}
and 
\begin{equation} 
r_{xwr,\cJ} = X_{\cJ}^{-1} r_{xs,\cJ} -r_{c_\cI}.
\end{equation}
Here $X_{\cI} = \diag(x_{\cI})$ and $X_{\cJ} = \diag(x_{\cI}).$
Let $D\in \RR^{n\times n}$ be a diagonal matrix such that 
\begin{equation}
    D_{\cI,\cI} = X_{\cI}^{-1}S_{\cI} +W^{-1}T,\quad \text{and}\quad  D_{\cJ,\cJ} = X_{\cJ}^{-1} S_{\cJ}. 
\end{equation}
Here $D_{\cI,\cI}$ is the block corresponding to the index set $\cI$ and $D_{\cJ,\cJ}$ is the block corresponding to the index set $\cJ$. 

After a series of eliminationn in \eqref{eq:lpupw.pdip}, we obtain $\Delta x$ and $\Delta \lambda$ by solving 
\begin{equation}\label{eq:lpupw.pdip.linearsolve1}
    \begin{bmatrix}
        -D &  A^\top \\
        A &  0 
    \end{bmatrix}
    \begin{bmatrix}
        \Delta x \\ 
        \Delta \lambda 
    \end{bmatrix} = 
    \begin{bmatrix}
        r_{vwr} \\ 
        -r_x
    \end{bmatrix}.
\end{equation}
One can then easily solve for $\Delta w$, $\Delta t$, and $\Delta s$. 
Reducing further, we obtain the system
\begin{equation}\label{eq:lpupw.pdip.linearsolve2}
    AD^{-1}A^\top \Delta \lambda = AD^{-1}r_{vwr} -r_x,
\end{equation}
which is the form often considered in production interior-point software.
We use \eqref{eq:lpupw.pdip.linearsolve1} in the experiments for the Netlib data set because this system yields more accurate solutions in a finite-precision setting than  \eqref{eq:lpupw.pdip.linearsolve2} \cite{wright1997stability, wright1999modified}.
We take steps of the form 
    $(x,w) +\alpha_P (\Delta x, \Delta w) $ and 
    $(\lambda, s,t) +\alpha _D (\Delta \lambda, \Delta s,\Delta t)$. We choose $\alpha_P$ and $\alpha_D$ similarly to \eqref{eq: pdipstepmaxlength} and \eqref{eq: actual_pdipstepmaxlength}. Precisely, we set the maximum step for primal and dual as 
\begin{equation*}
\begin{aligned}\label{eq: pdip_ub_stepmaxlength}
     \mbox{PDIP:} \quad  \alpha_P^{\max} & = \min\left( 1, \max\{\alpha\mid (x,w)+\alpha (\Delta x,\Delta w) \geq 0\}\right), \\
     \mbox{PDIP:} \quad  \alpha_D^{\max} & = \min\left( 1,\max\{\alpha\mid (s,t)+\alpha (\Delta s,\Delta t) \geq 0\}\right),
\end{aligned}
\end{equation*}
then choose the actual steplengths to be
\begin{equation*} \label{eq: pidp_ub_actual_pdipstepmaxlength}
    \mbox{PDIP:} \quad \alpha_P = \tau  \alpha_P^{\max},\quad \text{and}\quad 
    \alpha_D = \tau  \alpha_D^{\max},
\end{equation*}
for some parameter $\tau\in [0,1)$. 

\subsection{MPC} \label{sec: app.mpc}
The MPC updates the iterates $(x,w,\lambda,s,t)$ by solving the system \eqref{eq:lpupw.pdip} twice with a different right-hand side each time. We spell out the details here by following \cite[Chapter 10]{Wri97}. First, MPC obtained an affine-scaling direction $(\Delta x^{\text{aff}},\Delta w^{\text{aff}},\Delta \lambda^{\text{aff}}, \Delta s^{\text{aff}}, \Delta t^{\text{aff}})$ by solving \eqref{eq:lpupw.pdip} with $\sigma = 0$. Next, it computes primal and dual stepsize for these two directions:
\begin{equation*} \label{eq:jh1}
\begin{aligned}
  \alpha_P^{\text{aff}} & = \min\left( 1, \max\{\alpha\mid x+\alpha \Delta x^{\text{aff}} \geq 0, w+\alpha \Delta w^{\text{aff}}\geq 0 \}\right), \\
\quad  \alpha_D^{\text{aff}} & = \min\left( 1,\max\{\alpha\mid s+\alpha \Delta s^{\text{aff}} \geq 0,t+\alpha \Delta t^{\text{aff}}\geq 0\}\right).
\end{aligned}
\end{equation*}
After computing the affine-scaling stepsizes, it
computes a duality gap for the affine-scaling updates:
\begin{equation*}
    \mu_{\text{aff}} = 
    \frac{
    (x+\alpha_P^{\text{aff}}  
    \Delta x^{\text{aff}})^\top  (s+\alpha_D^{\text{aff}}\Delta s^{\text{aff}})+
    (w+\alpha_P^{\text{aff}} \Delta w^{\text{aff}})^\top (
    t+ \alpha_D^{\text{aff}} \Delta t^{\text{aff}} )
    }
    {
    n+|\cI|
    }.
\end{equation*}
Then it obtains the actual search direction $(\Delta x, \Delta w,\Delta \lambda, \Delta s,\Delta t)$ by 
solving the system \eqref{eq:lpupw.pdip} again with $\sigma$ set as follows: 
\begin{equation}
    \sigma = \left(\frac{\mu_{\text{aff}}}{\mu}\right)^3,\quad 
    \text{where}\quad \mu = \frac{x^\top s + w^\top t}{n+|\cI|}.
\end{equation}
It then sets the maximum stepsize for primal and dual as 
\begin{equation*}
\begin{aligned}\label{eq: mpcstepmaxlength}
     \mbox{MPC:} \quad  \alpha_P^{\max} & = \min\left( 1, \max\{\alpha\mid (x,w)+\alpha (\Delta x,\Delta w) \geq 0\}\right), \\
     \mbox{MPC:} \quad  \alpha_D^{\max} & = \min\left( 1,\max\{\alpha\mid (s,t)+\alpha (\Delta s,\Delta t) \geq 0\}\right),
\end{aligned}
\end{equation*}
and takes the actual step length to be
\begin{equation*} \label{eq: actual_mpcstepmaxlength}
    \mbox{MPC:} \quad \alpha_P = \tau  \alpha_P^{\max},\quad \text{and}\quad 
    \alpha_D = \tau  \alpha_D^{\max},
\end{equation*}
for some parameter $\tau\in [0,1)$. We then update the primal and dual variables as follows:
    $(x,w) +\alpha_P (\Delta x, \Delta w) $ and 
    $(\lambda, s,t) +\alpha _D (\Delta \lambda, \Delta s,\Delta t)$.

\subsection{SSV-SQP} \label{sec: app.ssv-sqp}
We first introduce the squared variables for \eqref{eq:lpupw} and obtain 
\begin{equation}\label{eq:lpupw.ssv} \tag{LP-uw-ssv}
\begin{aligned}
    \text{minimize}_{x\in \RR^n} &\quad    c^Tx  \\ 
    \text{subject to} & \quad Ax=b,  \\ 
                      & \quad u = w + x_{\cI} \\ 
                      & \quad x=v\odot v\quad \text{and} \quad w =y\odot y.
\end{aligned}
\end{equation}
KKT conditions for \eqref{eq:lpupw.ssv} is 
\begin{subequations} \label{eq:lpupw.ssv.kkt}
\begin{align}
\label{eq:lpupw.ssv.kkt.1a}
    A^T_{\cI} \lambda + s_\cI - t - c_\cI &= 0, \\
    \label{eq:lpupw.ssv.kkt.1b}
    A^T_{\cJ} \lambda + s_\cJ  - c_\cJ &= 0, \\
    \label{eq:lpupw.ssv.kkt.2a}
    Ax - b &=0, \\
    \label{eq:lpupw.ssv.kkt.2b}
    x_{\cI} +w &=u, \\
    \label{eq:lpupw.ssv.kkt.3}
    s \odot v = 0,\quad y\odot r =0 & \\
    \label{eq:lpupw.ssv.kkt.4}
    x = v\odot v, \quad w = y\odot y. &
\end{align}
\end{subequations}

Each iteration of SQP for \eqref{eq:lpupw.ssv.kkt} aims to solve:
\begin{subequations} \label{eq:lpupw.ssv.sqp}
\begin{align}
\label{eq:lpupw.ssv.sqp.1a}
    A^T_{\cI} \Delta \lambda + 
    \Delta s_\cI - \Delta t &= -r_{c_\cI} :\,= c_\cI -s_\cI +t -A_\cI \lambda , \\
    \label{eq:lpupw.ssv.sqp.1b}
    A^T_{\cJ} \Delta \lambda + 
    \Delta s_\cJ &= -r_{c_\cJ} :\,= c_\cJ -s_\cJ -A_\cJ \lambda \\
    \label{eq:lpupw.ssv.sqp.2a}
    A\Delta x  &=  -r_x:\, = b-Ax, \\
    \label{eq:lpupw.ssv.sqp.2b}
    \Delta x_{\cI} +\Delta w &=-r_u :\;= u -x_\cI -w, \\
    \label{eq:lpupw.ssv.sqp.3a}
    V\Delta s + S\Delta x  & = -r_{sv}:\;= - v\odot s,\\
     \label{eq:lpupw.ssv.sqp.3b}
    Y\Delta t + T \Delta W  & = -r_{ry}:\;= - t\odot y,\\
    \label{eq:lpupw.ssv.sqp.4a}
   \Delta x - 2v\odot \Delta v & = -r_v:\; = -x+v\odot v,\\
     \label{eq:lpupw.ssv.sqp.4b}
    \Delta w + 2 y \odot \Delta y & = -r_{y}:\;= -w + y\odot y.
 \end{align}
\end{subequations}
where $V = \diag(v)$, $S = \diag(s)$, $Y= \diag(y)$, and $T= \diag(t)$.
We maintain strict positivity of all components of $v$, $s$, $y$, and $t$ at all iterations.

Let $r_{vwr} \in \RR^n$ such that 
\begin{equation} 
r_{vwr,\cI} = V_{\cI}^{-1} r_{sv,\cI} -Y^{-1} r_{ry}+\tfrac{1}{2} V_{\cI}^{-2}S_{\cI} r_{v,\cI} -\tfrac{1}{2} Y^{-2}T r_y+\tfrac{1}{2}Y^{-2}T r_u -r_{c_\cI},
\end{equation}
and 
\begin{equation} 
r_{vwr,\cJ} = V_{\cJ}^{-1} r_{sv,\cJ} + \tfrac{1}{2}V_{\cJ}^{-2}S_{\cJ}r_{v,\cJ}-r_{c_\cJ}.
\end{equation}
Here $V_{\cI} = \diag(v_{\cI})$, $V_{\cJ} = \diag(v_{\cI})$, $S_{\cI} = \diag(s_{\cI})$, and $S_{\cJ} =\diag(s_{\cJ})$.
Let $D\in \RR^{n\times n}$ be the diagonal matrix whose diagonal blocks are 
\begin{equation}
    D_{\cI,\cI} = \tfrac12 V_{\cI}^{-2}S_{\cI} + \tfrac12 Y^{-2}T,\quad \text{and}\quad  D_{\cJ,\cJ} = \tfrac12 V_{\cJ}^{-2} S_{\cJ}. 
\end{equation}
Here $D_{\cI,\cI}$ is the block corresponding to the index set $\cI$ and $D_{\cJ,\cJ}$ is the block corresponding to the index set $\cJ$. 

After a series of elimination for \eqref{eq:lpupw.ssv.sqp}, we obtain $\Delta x$ and $\Delta \lambda$ by solving 
\begin{equation}\label{eq:lpupw.ssv.sqp.linearsolve1}
    \begin{bmatrix}
        -D &  A^\top \\
        A &  0 
    \end{bmatrix}
    \begin{bmatrix}
        \Delta x \\ 
        \Delta \lambda 
    \end{bmatrix} = 
    \begin{bmatrix}
        r_{vwr} \\ 
        -r_x
    \end{bmatrix}.
\end{equation}
One can recover the other step components.
Reducing further, we obtain the system 
\begin{equation}\label{eq:lpupw.ssv.sqp.linearsolve2}
    AD^{-1}A^\top \Delta \lambda = AD^{-1}r_{xwr} -r_x.
\end{equation}
We use \eqref{eq:lpupw.ssv.sqp.linearsolve1} in the experiments for the Netlib data set because this system yields more accurate solutions in a finite-precision setting than \eqref{eq:lpupw.ssv.sqp.linearsolve2}.
We update primal and dual variables via 
    $(x,w,y,v) +\alpha_P (\Delta x, \Delta w,\Delta y,\Delta v) $ and 
    $(\lambda, s,t) +\alpha _D (\Delta \lambda, \Delta s,\Delta t)$. The actual stepsizes $\alpha_P$ and $\alpha_D$ are chosen similarly to \eqref{eq: sqp_stepmaxlength} and \eqref{eq: actual_sqpstepmaxlength}. We set the maximum step for primal and dual as 
\begin{equation*}
\begin{aligned}\label{eq: ssv.ub.sqp_stepmaxlength}
    \mbox{SSV-SQP:} \quad \alpha_P^{\max} & = \min\left( 1, \max\{\alpha\mid (v,y)+\alpha (\Delta v,\Delta y) \geq 0\}\right), \\
    \mbox{SSV-SQP:} \quad  \alpha_D^{\max} & = \min\left( 1,\max\{\alpha\mid (s,t)+\alpha (\Delta s ,\Delta t) \geq 0\}\right),
\end{aligned}
\end{equation*}
with actual step lengths defined as
\begin{equation*} \label{eq: ssv.ub.actual_sqpstepmaxlength}
    \mbox{SSV-SQP:} \quad \alpha_P = \tau \alpha_P^{\max},\quad \text{and}\quad 
    \alpha_D = \tau \alpha_D^{\max},
\end{equation*}
where $\tau \in (0,1]$ is a fixed parameter.

\section{Results of SSV-SQP and MPC for each problem in Netlib}
\label{sec: netlib_MPCvsSSV} Table~\ref{tb: netlib_MPCvsSSV} shows the number of iterations for each problem in Netlib of SSV-SQP with $\tau =  0.5,\;0.75.\;0.9$ and MPC with $\tau = 0.9$ for accuracy parameters $10^{-2}$ and $10^{-5}$.
The problem dimensions $(m,n)$ are for the presolved versions, obtained from the presolver of PCx~\cite{czyzyk1999pcx}.

{
\small
{
\begin{longtable} {|c|c||*{3}{c}|c||*{3}{c}|c|} 
    \caption{The number of iterations of SSV-SQP (with $\tau = .5,.75,.9$) and MPC (with $\tau = .9$) for different problems in Netlib in terms of accuracy levels $\epsilon = 10^{-2}$ and $\epsilon = 10^{-5}$. For each value of $\epsilon$, we display the number of iterations for SSV-SQP (with $\tau = .5,.75,.9$) and MPC (with $\tau = .9$) from left to right. $*$ indicates the algorithm cannot attain the desired accuracy within $2000$ iterations or $1000$ seconds in terms of CPU time. For example, for $\epsilon = 10^{-2}$ and the instance 25fv47, we have 
    the number of iterations for SSV-SQP with $\tau = .5,.75,.9$ being 65, 66, and 57, and the number of iterations for MPC with $\tau = .9$ is 25. Problem names displayed with a superscript "$+$" are those for which SSV-SQP solves the problem successfully for $\epsilon=10^{-5}$ when $\tau=.15$, though unsuccessful for larger values of $\tau$.\label{tb: netlib_MPCvsSSV}}
    \\
     \hline
    \multicolumn{1}{|c|}{Name} & \multicolumn{1}{c||}{$(n,m)$} & 
     \multicolumn{4}{c||}{$\epsilon = 10^{-2}$} & \multicolumn{4}{c|}{$\epsilon=10^{-5}$}\\
     \hline 
     \hline
25fv47 & (1843,788) & 65 & 66 & 57 & 25 & 120 & $*$ & $*$ & 34 \\ 
\hline 
80bau3b$^+$ & (11066,2140) & 206 & 136 & 131 & $*$ & $*$ & $*$ & $*$ & $*$ \\ 
\hline 
adlittle & (137,55) & 44 & 29 & 24 & 18 & 54 & 35 & $*$ & 22 \\ 
\hline 
afiro & (51,27) & 34 & 21 & 17 & 14 & 47 & 29 & 23 & 17 \\ 
\hline 
agg2 & (750,514) & 93 & 63 & 52 & 27 & 109 & 76 & 67 & 33 \\ 
\hline 
agg3 & (750,514) & 86 & 60 & 52 & 26 & 102 & 71 & 68 & 32 \\ 
\hline 
agg & (477,390) & 68 & 48 & 43 & 25 & 80 & 56 & 51 & 31 \\ 
\hline 
bandm & (395,240) & 69 & 51 & 48 & 18 & 83 & 61 & $*$ & 24 \\ 
\hline 
beaconfd & (171,86) & 54 & 33 & 27 & 16 & 64 & 38 & 31 & 21 \\ 
\hline blend & (111,71) & 40 & 26 & 24 & 14 & 55 & 36 & $*$ & 18 \\ 
\hline 
bnl1$^+$ & (1491,610) & 54 & 39 & 35 & 28 & $*$ & $*$ & $*$ & 40 \\ 
\hline 
bnl2 & (4008,1964) & 103 & 125 & 154 & 30 & 214 & 185 & $*$ & 44 \\ 
\hline 
boeing1 & (697,331) & 49 & 37 & 35 & 19 & 99 & 69 & $*$ & 27 \\ 
\hline 
boeing2 & (264,125) & 60 & 49 & 49 & 20 & $*$ & $*$ & $*$ & 27 \\ \hline bore3d & (138,81) & 49 & 33 & 29 & 21 & 63 & 41 & 33 & 24 \\ \hline brandy & (238,133) & 56 & 39 & 37 & 19 & 82 & 59 & $*$ & $*$ \\ \hline capri & (436,241) & 100 & 84 & 92 & 23 & 117 & 97 & $*$ & 29 \\ \hline cycle & (2780,1420) & $*$ & 129 & $*$ & 33 & $*$ & $*$ & $*$ & $*$ \\ \hline czprob & (2779,671) & 51 & 37 & 34 & 33 & 65 & 46 & 41 & 40 \\ \hline d2q06c & (5728,2132) & 83 & 66 & 96 & 30 & 132 & $*$ & $*$ & $*$ \\ \hline d6cube & (5443,403) & 37 & 24 & 20 & 16 & 72 & 53 & 53 & 32 \\ \hline degen2 & (757,444) & 28 & 18 & 14 & 14 & 51 & 34 & 30 & 21 \\ \hline degen3 & (2604,1503) & 32 & 21 & 18 & 17 & 63 & $*$ & 41 & 25 \\ \hline dfl001 & (12143,5984) & $*$ & $*$ & $*$ & $*$ & $*$ & $*$ & $*$ & $*$ \\ \hline e226 & (429,198) & 41 & 26 & 22 & 17 & 67 & 48 & 43 & 24 \\ 
\hline 
etamacro$^+$ & (669,334) & 46 & 31 & 27 & 19 & $*$ & $*$ & $*$ & 29 \\ \hline fffff800 & (826,322) & 87 & 66 & 61 & 33 & 114 & $*$ & $*$ & 39 \\ 
\hline 
finnis$^+$ & (935,438) & 84 & 67 & 69 & 26 & $*$ & $*$ & $*$ & 31 \\ \hline fit1d & (1049,24) & 36 & 22 & 19 & 19 & 50 & 32 & 27 & 26 \\ \hline fit1p & (1677,627) & 94 & 85 & $*$ & 19 & 104 & 95 & $*$ & 24 \\ \hline fit2d & (10524,25) & 40 & 25 & 24 & 20 & 56 & 40 & 43 & 28 \\ \hline fit2p & (13525,3000) & 67 & 72 & 96 & 20 & 86 & $*$ & $*$ & 26 \\ \hline forplan & (447,121) & 53 & 33 & 31 & 22 & 82 & 57 & 53 & $*$ \\ \hline ganges & (1510,1113) & 52 & 35 & 32 & 18 & 96 & 68 & 63 & 23 \\ \hline gfrd-pnc & (1134,590) & 76 & 60 & 59 & 22 & 157 & 123 & 122 & 26 \\ 
\hline 
greenbea & (4164,1933) & 101 & $*$ & $*$ & 40 & $*$ & $*$ & $*$ & 47 \\ 
\hline 
greenbeb$^+$ & (4155,1932) & 107 & 83 & 83 & 42 & $*$ & $*$ & $*$ & $*$ \\ \hline grow15 & (645,300) & $*$ & $*$ & $*$ & 33 & $*$ & $*$ & $*$ & 38 \\ \hline grow22 & (946,440) & $*$ & $*$ & $*$ & 34 & $*$ & $*$ & $*$ & 39 \\ \hline grow7 & (301,140) & $*$ & $*$ & $*$ & 32 & $*$ & $*$ & $*$ & 36 \\ \hline israel & (316,174) & 85 & 66 & 64 & 25 & 107 & $*$ & $*$ & 29 \\ \hline kb2 & (68,43) & 37 & 26 & 23 & 15 & 51 & 32 & 26 & 19 \\ \hline lotfi & (346,133) & 69 & 50 & 48 & 17 & 87 & 63 & $*$ & 25 \\ \hline maros-r7 & (7440,2152) & 98 & 110 & 132 & 18 & 127 & 134 & 153 & 25 \\ \hline maros & (1437,655) & 56 & 35 & 28 & $*$ & 105 & $*$ & $*$ & $*$ \\ \hline modszk1 & (1599,665) & 168 & 133 & 109 & 22 & 280 & 267 & $*$ & 29 \\ 
\hline 
nesm$^+$ & (2922,654) & 51 & 33 & 37 & 24 & $*$ & $*$ & $*$ & 34 \\ 
\hline 
perold & (1462,593) & 118 & 83 & 73 & $*$ & 154 & $*$ & $*$ & $*$ \\ 
\hline 
pilot-ja$^+$ & (1892,810) & 88 & 65 & 63 & $*$ & $*$ & $*$ & $*$ & $*$ \\ 
\hline 
pilot-we & (2894,701) & 124 & 94 & 143 & $*$ & 145 & 121 & $*$ & $*$ \\ \hline pilot4 & (1110,396) & 132 & 96 & 88 & 24 & 183 & 134 & $*$ & 34 \\ \hline pilot87 & (6373,1971) & 140 & 112 & 93 & 27 & 275 & 267 & $*$ & 35 \\ \hline pilot & (4543,1368) & 95 & 67 & 59 & 30 & 215 & 170 & $*$ & 46 \\ \hline recipelp & (123,64) & 34 & 16 & 12 & 11 & 45 & 27 & 21 & 15 \\ \hline sc105 & (162,104) & 33 & 19 & 14 & 13 & 57 & 35 & 27 & 18 \\ \hline sc205 & (315,203) & 42 & 28 & 23 & 14 & 80 & 54 & 42 & 20 \\ \hline sc50a & (77,49) & 35 & 24 & 19 & 13 & 49 & 31 & 25 & 17 \\ \hline sc50b & (76,48) & 32 & 21 & 17 & 13 & 44 & 26 & 20 & 17 \\ \hline scagr25 & (669,469) & 52 & 35 & 30 & 19 & 64 & 43 & 36 & 24 \\ \hline scagr7 & (183,127) & 47 & 36 & 31 & 18 & 64 & 44 & 36 & 22 \\ \hline scfxm1 & (568,305) & 67 & 52 & 55 & 22 & 86 & 65 & $*$ & 28 \\ \hline scfxm2 & (1136,610) & 81 & 64 & 58 & 23 & 105 & 89 & $*$ & 31 \\ \hline scfxm3 & (1704,915) & 80 & 62 & 59 & 23 & 108 & 94 & 96 & 31 \\ \hline scorpion & (412,340) & $*$ & $*$ & $*$ & $*$ & $*$ & $*$ & $*$ & $*$ \\ \hline scrs8 & (1199,421) & 141 & 123 & $*$ & 22 & 198 & $*$ & $*$ & 30 \\ \hline scsd1 & (760,77) & 30 & 20 & 16 & 10 & 45 & $*$ & $*$ & 15 \\ \hline 
scsd6 & (1350,147) & 31 & 21 & 18 & 10 & 46 & 30 & $*$ & 17 \\ 
\hline 
scsd8$^+$ & (2750,397) & 32 & 22 & 18 & 12 & $*$ & $*$ & $*$ & 16 \\ 
\hline 
sctap1$^+$ & (644,284) & 56 & 40 & 39 & 18 & $*$ & $*$ & $*$ & 24 \\ 
\hline 
sctap2 & (2443,1033) & 50 & 38 & 41 & 17 & 128 & $*$ & $*$ & 22 \\ 
\hline 
sctap3$^+$ & (3268,1408) & 52 & 40 & 42 & 18 & $*$ & $*$ & $*$ & 23 \\ 
\hline 
seba & (901,448) & 53 & 37 & 33 & 20 & 68 & 45 & 39 & 23 \\ 
\hline 
share1b & (248,112) & 55 & 40 & 38 & 22 & 76 & 53 & $*$ & 28 \\ 
\hline 
share2b$^+$ & (162,96) & 33 & 22 & 18 & 15 & $*$ & $*$ & $*$ & 19 \\ 
\hline shell & (1451,487) & 93 & 78 & 79 & 31 & 106 & $*$ & $*$ & 35 \\ 
\hline 
ship04l & (1905,292) & 84 & 94 & 177 & 18 & 94 & 99 & $*$ & 22 \\ 
\hline 
ship04s & (1281,216) & 70 & 68 & 167 & 18 & 81 & 74 & $*$ & 22 \\ 
\hline 
ship08l & (3121,470) & 105 & 120 & 230 & 19 & 122 & 129 & $*$ & $*$ \\ 
\hline 
ship08s & (1604,276) & 65 & 55 & 125 & 18 & 77 & 60 & $*$ & $*$ \\ 
\hline 
ship12l & (4171,610) & 107 & 103 & 183 & 19 & 120 & $*$ & $*$ & 24 \\ 
\hline 
ship12s & (1943,340) & 65 & 59 & 74 & 18 & 79 & 67 & $*$ & 23 \\ 
\hline 
sierra & (2705,1212) & $*$ & $*$ & $*$ & $*$ & $*$ & $*$ & $*$ & $*$ \\ \hline 
stair & (538,356) & 44 & 26 & 19 & 16 & 79 & 49 & 40 & 23 \\ 
\hline 
standata & (796,314) & 61 & 46 & 52 & 23 & 80 & $*$ & $*$ & 27 \\ 
\hline 
standgub & (796,314) & 61 & 46 & 52 & 23 & 80 & $*$ & $*$ & 27 \\ 
\hline 
standmps$^+$ & (1192,422) & 66 & 49 & $*$ & 29 & $*$ & $*$ & $*$ & 42 \\ 
\hline 
stocfor1 & (150,102) & 40 & 26 & 22 & 18 & 55 & $*$ & $*$ & 22 \\ 
\hline 
stocfor2 & (2868,1980) & 71 & 60 & 62 & 22 & 112 & $*$ & $*$ & 29 \\ 
\hline 
stocfor3 & (22228,15362) & $*$ & $*$ & $*$ & 27 & $*$ & $*$ & $*$ & $*$ \\ 
\hline 
truss$^+$ & (8806,1000) & 45 & 34 & 31 & 16 & $*$ & $*$ & $*$ & 23 \\ 
\hline 
tuff & (567,257) & 45 & 30 & 26 & 18 & 111 & 78 & $*$ & 28 \\ 
\hline 
vtp-base & (111,72) & 44 & 30 & 25 & 17 & 54 & 35 & 28 & 20 \\ 
\hline 
wood1p & (1718,171) & 51 & 31 & 25 & 24 & 76 & 49 & 43 & 41 \\ 
\hline 
woodw & (5364,708) & 27 & 17 & 14 & 18 & 66 & 48 & 45 & 31 \\ 
\hline  
\end{longtable}
}
}

\end{document}